\documentclass[final,onefignum,onetabnum]{siamonline250211}

\usepackage{amssymb,amsfonts,mathtools}
\usepackage{algpseudocode}
\usepackage[numbers,sort&compress]{natbib}
\usepackage{enumitem}
\usepackage{needspace}

\makeatletter
\def\Url@OTnonTT{\do\<{\mathchar"313C }\do\>{\mathchar"313E }\do
  \_{\_\penalty\@m}\do\|{\mid}\do\{{\lbrace}\do\}{\mathbin{\rbrace}}\do
  \\{\mathbin{\backslash}}\UrlTildeSpecial}
\makeatother

\algrenewcommand\algorithmicrequire{\textbf{Input}}
\algrenewcommand\algorithmicensure{\textbf{Output}}

\newsiamremark{remark}{Remark}
\AddToHook{env/remark/begin}{\crefalias{theorem}{remark}}
\crefname{remark}{Remark}{Remarks}
\Crefname{remark}{Remark}{Remarks}

\DeclareMathOperator{\spanop}{span}
\DeclareMathOperator{\spec}{spec}
\DeclareMathOperator{\Ran}{Ran}
\DeclareMathOperator{\Lip}{Lip}
\DeclareMathOperator{\Ric}{Ric}
\newcommand{\dd}{\,\mathrm d}
\newcommand{\CC}{\mathbb C}
\newcommand{\EE}{\mathbb E}
\newcommand{\cB}{\mathcal B}
\newcommand{\cG}{\mathcal G}
\newcommand{\cH}{\mathcal H}
\newcommand{\cI}{\mathcal I}
\newcommand{\cJ}{\mathcal J}
\newcommand{\cK}{\mathcal K}
\newcommand{\cL}{\mathcal L}
\newcommand{\cQ}{\mathcal Q}
\newcommand{\cR}{\mathcal R}

\newcommand{\Fn}{\mathfrak F}
\DeclareMathOperator{\dist}{dist}
\DeclareMathOperator{\dom}{dom}
\DeclareMathOperator{\res}{res}
\newcommand{\Ba}{\mathbb B_a}
\newcommand{\nop}{n_{\rm op}}
\newcommand{\ip}[2]{\left\langle #1,#2\right\rangle}

\title{Learning Kolmogorov Backward Operators for Spectral Analysis and Forecasting: A Data-Driven Variational Galerkin Framework}
\author{Yuanchao Xu\thanks{Department of Mathematics, Graduate School of Science, Kyoto University, Kyoto 606-8502, Japan
  (\email{xu.yuanchao.3a@kyoto-u.ac.jp}, \email{ishikawa.isao.5s@kyoto-u.ac.jp}).}
\and Claire Valva\thanks{Department of Computing and Mathematical Sciences, California Institute of Technology, Pasadena, CA 91125, USA
  (\email{clairev@caltech.edu}).}
\and Ben Gao\thanks{Laboratoire Hubert Curien, UMR 5516, CNRS, IOGS, INRIA, Universit\'e Jean Monnet, 18 Rue du Professeur Beno{\^\i}t Lauras, 42000 Saint-\'Etienne, France
  (\email{ben.gao@univ-st-etienne.fr}).
  \funding{Y.X. and I.I. acknowledge support from JST CREST Grant No.~JPMJCR24Q1 and No.~JPMJCR24Q6, including the AIP Challenge Program. C.V. acknowledges support from the DoD Vannevar Bush Faculty Fellowship N00014-22-1-2790. B.G. acknowledges the Auvergne Rh\^one-Alpes region, Grenoble Alpes M\'etropole and BPIFrance, for project I-D\'emo R\'egion GreenAI public grant inside plan France 2030.}}
\and Isao Ishikawa\footnotemark[1]}
\headers{Backward Operators for Spectra and Forecasts}{Y. Xu, C. Valva, B. Gao, and I. Ishikawa}

\ifpdf
\hypersetup{
  pdftitle={Learning Kolmogorov Backward Operators for Spectral Analysis and Forecasting: A Data-Driven Variational Galerkin Framework},
  pdfauthor={Yuanchao Xu, Claire Valva, Ben Gao, and Isao Ishikawa}
}
\fi

\begin{document}
\maketitle

\begin{abstract}
For a time-homogeneous Markov process, the fixed-lag solution operator
\(
(T_\tau f)(x)=\mathbb E[f(X_{t+\tau})\mid X_t=x]
\)
of the Kolmogorov backward equation evolves observables linearly and
deterministically, even when the dynamics are stochastic. We approximate this operator using a two-stage method where we first approximate this operator directly from data, without knowledge of the governing equation (e.g., Fourier Neural Operator), and then compress the Galerkin projection onto a finite-dimensional observable space that is fixed in advance or selected variationally. The error analysis traces the learning error of the
estimator, the sampling error of the Galerkin inner products, and the
invariance and reconstruction properties of the observable space to the
computed spectra and forecasts. The sampling error enters only through a
multiplicative factor. Near an isolated spectral cluster, the invariance
defect of a subspace controls its distance from the target invariant
subspace. The forecast error is controlled by the one-step error of the
reduced model and the projection error of the observable; when the measure
is invariant and the reduced model preserves constants and is stable on
mean-zero observables, the bound is uniform in the horizon. For
deterministic dynamics, a learning rate for Fourier neural operators (FNOs)
gives one explicit bound for the learning error. Numerical experiments
illustrate spectral approximation and forecasting, including long-rollout
stability and a comparison with extended dynamic mode decomposition (EDMD)
when fewer transition pairs than basis functions are observed.
\end{abstract}

\begin{keywords}
Kolmogorov backward equation, data-driven dynamical systems, Koopman operator, operator learning, Galerkin approximation, invariant subspaces, spectral analysis, forecasting
\end{keywords}

\begin{MSCcodes}
37M10, 37M25, 47A15, 47A55, 47A58, 60J35, 68T07
\end{MSCcodes}

\section{Introduction}\label{sec:introduction}
When the governing equations are unknown, incomplete, or too costly to solve,
data-driven methods seek predictive and interpretable models directly from
time-resolved observations. Working with partial observations or stochastic systems complication this task as the measured variables need not satisfy a closed deterministic evolution law. 

A deterministic and linear description is available at the level
of observables. For a time-homogeneous Markov process, the fixed-lag solution
operator of the Kolmogorov backward equation is
\[
    (T_\tau f)(x)
    \coloneqq
    \EE\!\left[f(X_{t+\tau})\mid X_t=x\right].
\]
It remains linear and deterministic when the trajectories are random, and for
deterministic dynamics it reduces to the Koopman operator
\cite{koopman1931hamiltonian,koopman1932dynamical}. We refer to \(T_\tau\) as
the backward solution operator, or simply the solution operator. It is the
element of the transition semigroup at the fixed lag \(\tau\), not the
generator of the backward equation. Time-lagged data thus provide
statistical information about a deterministic evolution of observables.

Most data-driven approximations of this evolution are formed after an
observable space has been chosen. Extended dynamic mode decomposition (EDMD)
estimates the Galerkin compression of the operator on a fixed dictionary
\cite{williams2015data,korda2018convergence}. The dictionary may also be
learned \cite{li2017extended,takeishi2017learning,lusch2018deep,otto2019linearly},
for instance by minimizing single-transition spectral residuals
\cite{xu2025reskoopnet}, or it may be replaced by a reproducing-kernel space
\cite{klus2020eigendecompositions,kostic2022learning}, for which sharp
spectral error rates are available \cite{kostic2023sharp}. Variational
principles guide the choice of the approximation
\cite{noe2013variational,nuske2014variational,wu2020variational,mardt2018vampnets},
residual and invariance criteria assess its accuracy
\cite{colbrook2024rigorous,haseli2023invariance,conradie2026trustworthy},
and the resulting a posteriori error bounds have been minimized to select
dictionary and kernel parameters \cite{conradie2026trustworthy}. Several of
these ideas have been carried over to stochastic dynamics. Residual DMD has
been extended to account for the conditional variance of single transitions
\cite{colbrook2024beyond}. Invariant representations have been learned with
an objective whose maximizer is the leading invariant subspace
\cite{kostic2024learning}, and the Koopman semigroup has been approximated from
data through a stochastic Taylor expansion \cite{xu2025koopmansemigroup}.

In most of these methods, the approximation of the dynamics is tied to a
chosen set of coordinates and each candidate observable space requires its
own fit of the dynamics. Operator-learning methods instead approximate maps
between function spaces without first fixing a reduced representation
\cite{chen1995universal,lu2021learning,li2021fourier,kovachki2021universal,kovachki2023neural,stepaniants2023learning,xiong2024koopman,calvello2026operator}.
For overviews of the operator-theoretic formulation and its relation to
spectral decomposition, model reduction, and data-driven learning, see
\cite{mezic2005spectral,budivsic2012applied,mezic2021koopman}.

We combine these two viewpoints in a two-stage method, which we call the
backward-solution-operator (BSO) method. In the first stage, the solution
operator is approximated as a map on observables, before any reduced
coordinates are fixed. In the second stage, the estimator is projected onto
an observable space that is either chosen in advance or learned by
balancing approximate invariance against the representation of the
observables of interest. Approximating the operator before fixing the
coordinates has three consequences. First, the same estimator can be
evaluated on any candidate observable space, so the space can be searched
without refitting the dynamics. Second, for stochastic dynamics, an
estimator of \(T_\tau\) approximates the conditional mean \(T_\tau f\) of
each observable \(f\). The invariance of a candidate space is therefore
assessed without the conditional-variance contribution that enters
single-transition residuals. Third, once the estimator is fixed, its
\(m\)-dimensional Galerkin compression can be formed even when the number
\(\nop\) of observed transitions satisfies \(\nop<m\), whereas the
corresponding EDMD Gram matrix is then necessarily singular.

The main theoretical contribution of this work is an error analysis of the
entire two-stage method, from the approximation of \(T_\tau\) to the
computed spectra and forecasts, in which the observable space may itself be
selected from the data. The analysis separates three sources of error: the
learning error of the estimator on the observable space, the sampling error
of the empirical inner products used to assemble the Galerkin matrix, and
the invariance and reconstruction properties of the observable space
itself. The sampling error enters only through a factor that tends to one.
The main results are the following.
\begin{itemize}
\item The space selected by the variational objective is nearly as good as
the best candidate when measured by the exact objective, up to the
learning error, the optimization tolerance, and the sampling factor
(\cref{thm:main}). The computed eigenvalues and spectral subspaces
converge whenever the observable spaces approach a finite-dimensional
invariant subspace, the learning error on these spaces tends to zero, and
the sampling error stays bounded away from one
(\cref{thm:galerkin-spectral}). If the corresponding spectral cluster is
isolated, the invariance defect of a subspace is locally equivalent to its
distance from the invariant subspace
(\cref{thm:quantitative-identification}). The variational objective
therefore bounds this distance, provided that the reconstruction
observables resolve the subspace and lie sufficiently close to it and that
the value of the objective is small relative to its reconstruction weight.
\item The forecast error is bounded by the one-step error of the space in
use and by the part of the observable that this space does not represent
(\cref{thm:forecast}); no invariant subspace is required. For an
invariant measure and reduced models that preserve constants and are
stable on mean-zero observables, a bound that is uniform in the horizon is
given in \cref{app:forecast-supplements}.
\item For deterministic dynamics, a learning rate for Fourier neural
operators (FNOs) provides one explicit bound for the first stage
(\cref{cor:fno-rate}); it can be inserted into the spectral and forecast
estimates above.
\end{itemize}
The analysis does not depend on how the estimator is constructed, and
stochastic dynamics enter only through its learning error. For stochastic
dynamics, the estimator has to approximate conditional means, for instance
by regression or through a learned transition kernel
(\cref{rem:neural-operator}); no explicit learning rate is given for that
case.

The rest of the paper is organized as follows. \Cref{sec:preliminaries}
gives the operator and spectral setting. \Cref{sec:methodology} presents
the method, its computational realization, and its relation to EDMD.
\Cref{sec:variational} establishes the error analysis, and
\cref{sec:numerical-experiments} reports the numerical experiments.
\Cref{sec:conclusion} concludes. Proofs and supplementary material are
collected in the appendices.

\section{Preliminaries}
\label{sec:preliminaries}

This section fixes the notation and recalls the operator and spectral
setting used throughout the paper.

\paragraph{Notation}
Unless stated otherwise, observables belong to
\(
    \cH\coloneqq L^2(\Omega,\rho;\mathbb R),
\)
where \(\Omega\) is a standard Borel space and \(\rho\) is a Borel
probability measure on \(\Omega\). The symbols \(\ip{\cdot}{\cdot}\) and
\(\|\cdot\|\) denote the \(L^2(\rho)\) inner product and norm. For normed
spaces \(\mathcal X\) and \(\mathcal Y\), \(\cB(\mathcal X,\mathcal Y)\)
denotes the bounded linear operators from \(\mathcal X\) to \(\mathcal Y\)
with the operator norm \(\|\cdot\|\), and \(\cB(\mathcal X)\coloneqq
\cB(\mathcal X,\mathcal X)\). For a closed subspace \(V\) of \(\cH\),
\(P_V\) is the orthogonal projection onto \(V\).

The complexification of a real linear space \(W\) of functions or of
equivalence classes is \(W+iW\), and that of \(\cH\) carries the inner
product that is conjugate-linear in its first argument. Whenever complex
scalars appear, a real space and its complexification are denoted by the
same symbol, and so are a real operator and its complex-linear extension,
which has the same norm; spectra always refer to the extension.

For vectors and matrices, \({}^\top\) denotes the transpose, \({}^*\) the
conjugate transpose, \({}^\dagger\) the Moore--Penrose pseudoinverse, and
\(\|\cdot\|_2\) the Euclidean norm of a vector and the spectral norm of a
matrix; \(\lambda_{\max}\) and \(\lambda_{\min}\) denote the largest and
smallest eigenvalues of a Hermitian matrix. For operators between Hilbert
spaces, \({}^*\) denotes the adjoint. A basis of a finite-dimensional
subspace that is orthonormal in \(L^2(\rho)\) is called \(\rho\)-orthonormal
and is denoted by \(\chi\), with a superscript naming the subspace when it
is not \(V\). A hat denotes a quantity computed from data, and a subscript
\(\star\) refers to the target invariant subspace \(E_\star\) introduced in
\cref{sec:spec_pert_theory}.

\subsection{Kolmogorov backward equation}\label{sec:kbe}
Let \((X_t)_{t\ge0}\) be a time-homogeneous Markov process on \(\Omega\)
with transition semigroup \((T_t)_{t\ge0}\),
\[
    (T_t f)(x)\coloneqq\EE\bigl[f(X_t)\mid X_0=x\bigr].
\]
Time homogeneity gives
\((T_\tau f)(x)=\EE[f(X_{s+\tau})\mid X_s=x]\) for every \(s\ge0\). The
operator \(T_\tau\) is an element of the Markov semigroup
\cite{pavliotis2014stochastic} and is also called the stochastic Koopman
operator
\cite{crnjariczic2020koopman,oprea2025distributional,xu2025koopmansemigroup}.
If \(\rho\) is invariant, then \(T_\tau\) is a contraction on \(L^2(\rho)\)
by Jensen's inequality and \(T_\tau1=1\); otherwise we assume that
\(T_\tau\) is well defined and bounded on \(L^2(\rho)\). When the semigroup
is strongly continuous on \(L^2(\rho)\), it has a generator \(\cL\) with
domain \(\dom(\cL)\), and \(T_t=e^{t\cL}\) in the semigroup sense. For
\(f\in\dom(\cL)\), the function \(u(t,\cdot)=T_tf\) then solves the
Kolmogorov backward equation \cite{pavliotis2014stochastic,engel2000one}
\[
    \partial_t u=\cL u,
    \qquad
    u(0,\cdot)=f.
\]
For an It\^o diffusion
\(\dd X_t=b(X_t)\dd t+\sigma(X_t)\dd W_t\) with drift \(b\) and diffusion
coefficient \(\sigma\), driven by a Brownian motion \(W_t\), It\^o's formula
\cite{pavliotis2014stochastic} gives, for smooth \(f\),
\[
    \cL f
    =b\cdot\nabla f
    +\frac12\operatorname{tr}(\sigma\sigma^\top\nabla^2 f),
\]
while for \(\dot x=F(x)\), \(\cL f=F\cdot\nabla f\). For deterministic
dynamics with flow map \(\Phi^t\), \(X_t=\Phi^t(X_0)\) and
\(T_t f=f\circ\Phi^t\). Throughout the analysis, the lag \(\tau>0\) is
fixed, and we write
\[
    T\coloneqq T_\tau,
    \qquad
    M_T\coloneqq\|T\|_{\cB(\cH)}.
\]

\subsection{Invariant subspaces and spectral projections}\label{sec:spec_pert_theory}
The targets of the spectral computation are finite-dimensional
\(T\)-invariant subspaces, and computed spaces are compared with them
through the distance
\(
    \dist(V,W)\coloneqq\|P_V-P_W\|
\)
between closed subspaces \(V\) and \(W\) of \(\cH\).
If \(V\) and \(W\) have the same finite dimension, then
\(\dist(V,W)=\|(I-P_V)P_W\|=\|(I-P_W)P_V\|\) is the sine of the largest
principal angle between \(V\) and \(W\)
\cite{bjorck1973numerical,stewart1990matrix},
\cite[Chap.~I, \S6.8]{kato1966perturbation}.

A subspace \(E\subset\cH\) is \(T\)-invariant if \(TE\subset E\). Throughout,
\(E_\star\) denotes an \(m\)-dimensional \(T\)-invariant subspace and
\(P_\star\coloneqq P_{E_\star}\); the eigenvalues of \(T|_{E_\star}\) are
eigenvalues of \(T\). The main examples are spectral subspaces. Suppose that
a finite, conjugation-invariant set \(\Lambda_\star\subset\spec(T)\), an
isolated cluster, is separated from \(\spec(T)\setminus\Lambda_\star\) by a
positively oriented contour \(\Gamma\) in the resolvent set, and that the
Riesz projection
\[
    \Pi
    \coloneqq
    \frac{1}{2\pi i}\oint_\Gamma(zI-T)^{-1}\,\dd z
\]
has finite rank \(m\). Since \(T\) is real and \(\Lambda_\star\) is
conjugation invariant, so is \(\Ran\Pi\), and its real part
\(
    E_\star=\{u\in\Ran\Pi:\overline u=u\}
\)
is an \(m\)-dimensional \(T\)-invariant subspace with
\(\Ran\Pi=E_\star+iE_\star\) and \(\spec(T|_{E_\star})=\Lambda_\star\)
\cite[Chap.~III, \S6.4]{kato1966perturbation}. The complexification
\(\Ran\Pi\) is the spectral subspace of \(\Lambda_\star\), and we also call
\(E_\star\) the spectral subspace of \(\Lambda_\star\). In reversible
diffusion settings, \(\cL\) is self-adjoint, and \(T_\tau=e^{\tau\cL}\) and
\(\cL\) have the same eigenspaces by the spectral theorem; in the
non-self-adjoint case, spectral subspaces are the stable spectral objects
\cite{stewart1990matrix,kato1966perturbation}. The eigenvalues \(\mu\) of
\(T_\tau\) are discrete-time multipliers. When continuous-time rates are
reported, \(\lambda=\tau^{-1}\log\mu\), \(\mu\ne0\), is evaluated on a fixed
branch of the logarithm.

\section{Approximation and Galerkin projection of the solution operator}\label{sec:methodology}

The BSO method has two stages. First, an estimator \(\widehat T\) of the
solution operator \(T\) is constructed from data. Second, the estimator is
compressed to a Galerkin matrix on an observable space \(V\) that is fixed in
advance or selected variationally. In matrix form, the second stage is a
least-squares problem. Let \(\psi_1,\ldots,\psi_m\) be a basis of \(V\), let
\(x_1,\ldots,x_n\) be sample points in \(\Omega\), and define the data
matrices
\[
    (\mathbf\Psi)_{kj}\coloneqq\psi_j(x_k),
    \qquad
    (\mathbf\Psi_{\widehat T})_{kj}\coloneqq(\widehat T\psi_j)(x_k),
    \qquad
    \mathbf\Psi,\mathbf\Psi_{\widehat T}\in\mathbb R^{n\times m},
\]
where boldface distinguishes data matrices from bases. If \(\mathbf\Psi\)
has full column rank, the BSO matrix is
\begin{equation}\label{eq:bso-least-squares}
    \widehat K
    \coloneqq
    \mathbf\Psi^\dagger\mathbf\Psi_{\widehat T}
    =
    \bigl(\mathbf\Psi^\top\mathbf\Psi\bigr)^{-1}
    \mathbf\Psi^\top\mathbf\Psi_{\widehat T}.
\end{equation}
This is the EDMD matrix with the observed time-advanced values replaced by the
values \((\widehat T\psi_j)(x_k)\) of the estimator, so that the sample
points need not have observed next-step values. The eigenpairs of \(\widehat K\)
are used for spectral analysis and its powers for forecasting. For the
analysis, the basis is orthonormalized in the empirical inner product of the
sample points (\cref{sec:computational-realization}), which turns
\eqref{eq:bso-least-squares} into a matrix of inner products. For a linear
estimator, this changes \(\widehat K\) only by a similarity transformation.
The two stages are described in
\cref{sec:data-to-backward,sec:galerkin-method}, the computational
realization in \cref{sec:computational-realization}, and the relation to
EDMD in \cref{subsec:smalldata}.

\subsection{Learning the solution operator}
\label{sec:data-to-backward}

Let \(X\) be a state observed at a time \(t\ge0\), whose law is absolutely
continuous with respect to \(\rho\), and let \(X^+\coloneqq X_{t+\tau}\) be
the state one lag later. By the definition of \(T\) in \cref{sec:kbe},
every observable \(f\) satisfies
\[
    (Tf)(X)
    =
    \EE\!\left[f(X^+)\mid X\right]
    \qquad\text{almost surely},
\]
so that the time-lagged values of every observable are governed by \(T\).
The data \(\mathcal D_{\rm op}\) of the method are \(\nop\) time-lagged
pairs of observables, which arise in two forms. In the first form, they are
the lagged values \(f(X_k)\) and \(f(X_k^+)\) of observables \(f\) at
observation times \(t_1,\ldots,t_{\nop}\), where \(X_k\coloneqq X_{t_k}\)
and \(X_k^+\coloneqq X_{t_k+\tau}\); for the observed variables of
\cref{rem:neural-operator}, these lagged values are the lagged observations
\((y_{t_k},y_{t_k+\tau})\). In the second form, used in the spectral
experiments of \cref{sec:numerical-experiments}, they are observables
\(f_k\) together with their images \(Tf_k\) under the backward evolution
over one lag. From these data an estimator \(\widehat T\) of \(T\) is
constructed without knowledge of the governing equation. The analysis is
independent of the operator-learning method and depends only on the
accuracy of \(\widehat T\) on the observable spaces under consideration.
With partial observations, the same procedure approximates a compression
of \(T\) to the functions of the observed variables
(\cref{app:partial-observation}).

One technical point concerns the functions on which \(\widehat T\) acts.
Only the values of \(\widehat T\) on finitely many functions --- evaluated at
the sample points of the Galerkin step --- enter the computation. Point
evaluations are not defined on the \(\rho\)-equivalence classes that form
\(\cH\). We therefore fix a linear space \(\Fn\) of everywhere-defined
measurable functions \(\Omega\to\mathbb R\) that are square integrable with
respect to \(\rho\), and we regard \(\widehat T\) as a linear map
\(\widehat T\colon\Fn\to\Fn\); estimators that are not linear are covered
by \cref{rem:nonlinear-estimator}. The space \(\Fn\) is not canonical. Any
such space may be used that contains the constant function \(1\), the
candidate dictionaries, and the reconstruction observables of
\cref{sec:galerkin-method}, and that is mapped into itself by
\(\widehat T\). The conventions for passing between functions in \(\Fn\)
and their \(\rho\)-equivalence classes are collected in
\cref{app:conventions}.

\begin{remark}\label{rem:neural-operator}
In the experiments of \cref{sec:numerical-experiments}, \(\widehat T\) is a
neural operator, obtained in one of two ways. In the spectral examples, a
neural operator is trained on pairs \((f,Tf)\) of observables and
approximates \(T\) directly. Such an estimator is in general nonlinear in
\(f\) and is covered by \cref{rem:nonlinear-estimator}. In the forecasting
examples, the data are time series of observed variables \(y_t=h(X_t)\)
with values in a space \(Y\), such as sensor readings, fields, or delay
vectors, where each component of \(h\) is an observable. A neural operator
\(\widehat\Phi\colon Y\to Y\) is then trained on the lagged observations
\((y_t,y_{t+\tau})\). If the observed variables evolve autonomously,
\(y_{t+\tau}=\Phi(y_t)\), as for deterministic dynamics observed through an
injective map or through delay coordinates, then \(T\) acts on observables
of \(y\) by composition with \(\Phi\); if the full state is observed,
\(\Phi=\Phi^\tau\). We then take \(\Omega=Y\), with \(\rho\) a probability
measure on \(Y\) and \(Tf=f\circ\Phi\), and use the composition estimator
\(\widehat Tf\coloneqq f\circ\widehat\Phi\), which is linear in \(f\). If
the observed variables do not evolve autonomously, as for stochastic
dynamics or partial observations, and \(\widehat\Phi\) is trained by least
squares, then \(f\circ\widehat\Phi\) approximates the conditional mean of
\(f(y_{t+\tau})\) given \(y_t\) for affine \(f\), but in general not for
nonlinear \(f\). In that case the estimator is obtained by regressing
\(f(y_{t+\tau})\) on \(y_t\) or from an estimate of the transition kernel
(\cref{app:partial-observation}). Other operator-learning and regression
methods can be used in the same way.
\end{remark}

\subsection{Candidate spaces and the variational objective}
\label{sec:galerkin-method}

Throughout, \(m\ge1\) denotes the dimension of the observable spaces; the
invariant subspaces \(E_\star\) of \cref{sec:spec_pert_theory} to which they
are compared have the same dimension. Fix a parameter set \(\Theta\), and
let each
\(\theta\in\Theta\) define a raw dictionary
\(\phi_\theta=(\phi_{\theta,1},\ldots,\phi_{\theta,m})\) of functions in
\(\Fn\) whose \(\rho\)-equivalence classes are linearly independent, with
span \(V_\theta\). Two examples are used below. In a neural dictionary on a
subset \(\Omega\) of a Euclidean space, \(\phi_\theta\colon\Omega\to\mathbb
R^m\) is a network with one hidden \(\tanh\) layer and \(m\) outputs, and
\(\theta\) collects its weights and biases, which range over a bounded set.
In a polynomial dictionary, the \(\phi_{\theta,j}\) are polynomials of
bounded degree, and \(\theta\) collects their coefficients. A dictionary
that is fixed in advance corresponds to a singleton \(\Theta\), and the
constant function may be kept as a fixed element of every dictionary.
Families \(\Theta_p\), \(p\in\mathbb N\), of increasing richness are
considered in \cref{subsec:dictionary-observable-accuracy} and
\cref{rem:level}; the level \(p\) is, for instance, the network width or
the polynomial degree. Quantities such as \(\Delta_{\Theta_p}\) and
\(\eta_{\Theta_p}\) are those defined below for a parameter set \(\Theta\),
taken with \(\Theta=\Theta_p\).

The observables to be represented, called reconstruction observables, form
a finite set \(\cG\coloneqq\{g_1,\ldots,g_q\}\subset\Fn\); the observables
to be forecast in \cref{subsec:forecast-estimates} are taken from their
span. For an \(m\)-dimensional subspace
\(V\subset\cH\), define the invariance defect and the reconstruction error
\[
    \cI(V)\coloneqq\|(I-P_V)T|_V\|_{\cB(V,\cH)},
    \qquad
    \cR(V)\coloneqq\max_{1\le j\le q}\|(I-P_V)g_j\|,
\]
and, for a weight \(\alpha>0\), the objective
\[
    \cJ(V)\coloneqq\cI(V)^2+\alpha\,\cR(V)^2.
\]
The invariance defect \(\cI(V)\) vanishes if and only if \(V\) is
\(T\)-invariant; it measures the part of \(TV\) that is not represented in
\(V\). The reconstruction error \(\cR(V)\) measures the part of the
reconstruction observables that is not represented in \(V\), and the
weight \(\alpha\) balances the two terms.

For an \(m\)-dimensional subspace \(V\subset\Fn\) on which \(\|\cdot\|\) is
a norm, such as \(V_\theta\), the estimator enters the analysis through its
learning error on \(V\),
\begin{equation}\label{eq:learning-error}
    \Delta(V)\coloneqq\bigl\|\widehat T|_V-T|_V\bigr\|_{\cB(V,\cH)},
\end{equation}
where \(V\) is identified with its image in \(\cH\)
(\cref{app:conventions}); it is controlled by the accuracy of \(\widehat T\)
on the functions of any basis of \(V\) (\cref{lem:basis-linear-surrogate}).

\begin{remark}[Nonlinear estimators]\label{rem:nonlinear-estimator}
An estimator that is not linear in \(f\), such as a neural operator trained
on pairs \((f,Tf)\), enters the computation only through its values on
the computational basis
\(\widehat\Psi=(\widehat\psi_1,\ldots,\widehat\psi_m)\) of \(V\)
(\cref{sec:computational-realization}). For such an estimator,
\(\widehat T|_V\) denotes the linear map \(V\to\Fn\) that sends
\(\widehat\psi_j\) to \(\widehat T\widehat\psi_j\), and \(\widehat TV\)
denotes its range; with this convention, all results below remain valid, and
\cref{lem:basis-linear-surrogate} bounds the learning error by the errors of
\(\widehat T\) on the basis functions. The map \(\widehat T|_V\), and with
it the computed Galerkin operator of \cref{sec:computational-realization},
then depends on the computational basis, hence on the orthonormalization
rule and on the sample. If \(\widehat T|_V\) and \(\widehat T|_V'\) are the maps defined by
two bases of \(V\), with learning errors \(\Delta(V)\) and \(\Delta'(V)\),
then
\begin{equation}\label{eq:basis-dependence}
    \bigl\|\widehat T|_V-\widehat T|_V'\bigr\|_{\cB(V,\cH)}
    \le
    \bigl\|\widehat T|_V-T|_V\bigr\|+\bigl\|T|_V-\widehat T|_V'\bigr\|
    =\Delta(V)+\Delta'(V),
\end{equation}
and each of the two computed Galerkin operators satisfies the bound
\eqref{eq:compression-error} of \cref{prop:one-step} with respect to the
same exact compression \(P_VT|_V\), with its own learning error and its
own sampling error; the sampling errors differ because the range
\(\widehat TV\) depends on the basis. The bounds of \cref{sec:variational}
therefore hold for either choice of the basis, and by
\eqref{eq:basis-dependence} the two choices give nearby operators whenever
both learning errors are small.
\end{remark}

\begin{remark}\label{rem:stochastic-residual-bias}
Let \(X\sim\rho\) and let \(X^+\) denote the state one lag \(\tau\) after
\(X\). For \(\varphi\in\cH\) with \(\EE\,|\varphi(X^+)|^2<\infty\) and
\(\mu\in\CC\), a single-transition residual satisfies
\[
    \EE\,|\varphi(X^+)-\mu\varphi(X)|^2
    =\|T\varphi-\mu\varphi\|^2
     +\EE\,|\varphi(X^+)-(T\varphi)(X)|^2 ,
\]
because \(\varphi(X^+)-(T\varphi)(X)\) has conditional mean zero given
\(X\) and is therefore uncorrelated with \((T\varphi-\mu\varphi)(X)\).
Thus residuals formed from individual stochastic transitions contain an
irreducible conditional-variance term, the last term above, whereas
conditional-mean estimation removes this contribution. The same phenomenon
is known as the double-sampling issue in Bellman-residual minimization in
reinforcement learning \cite{baird1995residual,antos2008learning}.
\end{remark}

\subsection{Computational realization}
\label{sec:computational-realization}

Let \(x_1,\ldots,x_n\) be the sample points of
\eqref{eq:bso-least-squares}, either drawn independently from \(\rho\) or
taken from a stationary ergodic trajectory with invariant law \(\rho\). The
number \(n\) of sample points used for the Galerkin inner products is
distinct from the number \(\nop\) of pairs used to construct
\(\widehat T\), and the sample points need not be states at which data
were observed (\cref{subsec:smalldata}). For convergence results for DMD
based on ergodic trajectory data and Hankel-type observables, see
\cite{arbabi2017ergodic}. The empirical inner product and seminorm are
\[
    \ip{u}{v}_n
    \coloneqq
    \sum_{k=1}^n
    w_k\,\overline{u(x_k)}\,v(x_k),
    \qquad
    \|u\|_n\coloneqq\ip{u}{u}_n^{1/2},
    \qquad
    u,v\in\Fn,
\]
with the weights \(w_k=1/n\). More generally, the points \(x_k\) may be the
nodes of a quadrature rule for \(\rho\) with positive weights \(w_k\) summing
to one. The deterministic statements below hold verbatim for such weights,
with \eqref{eq:bso-least-squares} replaced by the corresponding weighted
least-squares matrix; the sampling error \(\eta\) of \cref{subsec:sampling}
is then the quadrature error of the rule.

Let \(V=V_\theta\) for some \(\theta\in\Theta\). Assume that the empirical
Gram matrix \([\langle\phi_{\theta,i},\phi_{\theta,j}\rangle_n]_{i,j=1}^m\)
of \(\phi_\theta\) is nonsingular, which requires \(n\ge m\). Let
\(\widehat\Psi=(\widehat\psi_1,\ldots,\widehat\psi_m)\) be the basis of
\(V\) obtained by orthonormalizing \(\phi_\theta\) in
\(\ip{\cdot}{\cdot}_n\) by a fixed rule, for example Gram--Schmidt or
symmetric orthonormalization; we call \(\widehat\Psi\) the computational
basis of \(V\). For a subspace \(V\) on which \(\ip{\cdot}{\cdot}_n\) is an
inner product, \(\widehat P_V\) denotes the
\(\ip{\cdot}{\cdot}_n\)-orthogonal projection onto \(V\), which is defined
on \(\Fn\). From now on, \(\widehat K\) denotes the least-squares matrix
\eqref{eq:bso-least-squares} for the computational basis. Since
\(\mathbf\Psi^\top\mathbf\Psi=nI\) for this basis, this computed Galerkin
matrix and the image Gram matrix are
\[
    \widehat K
    =
    \bigl[
        \ip{\widehat\psi_i}
            {\widehat T\widehat\psi_j}_n
    \bigr]_{i,j=1}^m,
    \qquad
    \widehat L
    \coloneqq
    \bigl[
        \ip{\widehat T\widehat\psi_i}
            {\widehat T\widehat\psi_j}_n
    \bigr]_{i,j=1}^m.
\]
The matrix \(\widehat L\) is the counterpart of the additional matrix that
residual DMD adds to those of EDMD \cite{colbrook2024rigorous}, with the
values of the estimator in place of the observed values. The matrix
\(\widehat K\) is the matrix, in the basis \(\widehat\Psi\), of the computed
Galerkin operator
\begin{equation}\label{eq:computed-galerkin-operator}
    \widehat\cK
    \coloneqq
    \widehat P_V\widehat T|_V
    \in\cB(V),
    \qquad
    \widehat\cK\widehat\psi_j
    =\sum_{i=1}^m\widehat K_{ij}\,\widehat\psi_i.
\end{equation}
The eigenvalues of \(\widehat K\) are those of \(\widehat\cK\), and an
eigenvector \(c\) of \(\widehat K\) corresponds to the eigenfunction
\(\sum_j c_j\widehat\psi_j\) of \(\widehat\cK\). The computed objective of
\(V=V_\theta\) is
\begin{equation}\label{eq:computed-objective}
    \widehat\cJ(\theta)
    \coloneqq
    \lambda_{\max}\!\left(
        \widehat L
        -
        \widehat K^\top
        \widehat K
    \right)
    +
    \alpha
    \max_{1\le j\le q}
    \|(I-\widehat P_V)g_j\|_n^2 .
\end{equation}
When \(\ip{\cdot}{\cdot}_n\) is an inner product on \(V+\widehat TV\), its
first term is the squared invariance defect of \(V\) with respect to
\(\widehat T\), measured in \(\ip{\cdot}{\cdot}_n\)
(\cref{lem:basis-linear-surrogate}). Thus \(\widehat\cJ(\theta)\) is the
empirical counterpart of \(\cJ(V_\theta)\), with \(T\) replaced by
\(\widehat T\).

\begin{algorithm}[tbp]
\caption{Fixed-lag solution-operator approximation and Galerkin projection}
\label{alg:variational-spectral-learning}
\begin{algorithmic}[1]

\Require Data \(\mathcal D_{\rm op}\) for constructing
\(\widehat T\) (\cref{sec:data-to-backward}); sample points
\(x_1,\ldots,x_n\), with their weights if these are not equal; either a
fixed dictionary or
candidate dictionaries \(\{\phi_\theta\}_{\theta\in\Theta}\); the
reconstruction observables \(\cG\), and for space selection \(\alpha>0\) and
a tolerance \(\zeta\ge0\); optional observable \(f\in\spanop\cG\) and
horizon \(N\).

\Ensure Space \(\widehat V\), basis \(\widehat\Psi\), Galerkin matrix
\(\widehat K\), the requested eigenvalues and associated eigenvectors or
Schur vectors, and optional forecast.

\State Construct \(\widehat T\) from \(\mathcal D_{\rm op}\).

\If{the dictionary is fixed in advance}
    \State Set \(\widehat V\) equal to its span and \(\widehat\Psi\) equal to
    its basis orthonormalized in \(\ip{\cdot}{\cdot}_n\).
\Else
    \State Compute \(\widehat\theta\in\Theta\) with
    \(
        \widehat\cJ(\widehat\theta)
        \le
        \inf_{\theta\in\Theta}\widehat\cJ(\theta)+\zeta,
    \)
    and set
    \(
        \widehat V\gets V_{\widehat\theta}
    \)
    and \(\widehat\Psi\) equal to its computational basis.
\EndIf

\State Evaluate
\(\widehat T\widehat\psi_j\), \(j=1,\ldots,m\), and set
\(
    \widehat K
    \gets
    \bigl[
        \ip{\widehat\psi_i}
            {\widehat T\widehat\psi_j}_n
    \bigr]_{i,j=1}^m,
\)
the matrix of the computed Galerkin operator
\eqref{eq:computed-galerkin-operator}.

\State Compute the required eigenvalues, eigenvectors, or Schur
vectors of \(\widehat K\).

\If{\(f\) and \(N\) are supplied}
    \State Compute
    \(\widehat c\coloneqq(\ip{\widehat\psi_j}{f}_n)_{j=1}^m\), the
    coefficients of \(\widehat P_{\widehat V}f\), and return the forecast
    \(\sum_{j}(\widehat K^N\widehat c)_j\widehat\psi_j\).
\EndIf

\end{algorithmic}
\end{algorithm}

\Cref{alg:variational-spectral-learning} summarizes the procedure. Training
the estimator \(\widehat T\) is the main computational cost, while the
Galerkin step is a small least-squares problem that can be repeated over
candidate spaces at negligible cost. Operator
estimation, space selection, and matrix assembly may use separate data
sets; the effect of reusing data on the error bounds is described in
\cref{subsec:sampling} and \cref{app:empirical-concentration}.

\subsection{Relation to EDMD}
\label{subsec:smalldata}
Consider a fixed space \(V\) and data given by lagged values at the states
\(X_k\) and \(X_k^+\) of \cref{sec:data-to-backward}; we call
\((X_k,X_k^+)\) transition pairs in what follows. Take the states \(X_k\)
as the sample points \(x_k\) of \cref{sec:computational-realization}.
Replacing \(\mathbf\Psi_{\widehat T}\) in \eqref{eq:bso-least-squares} by
the matrix of observed values \(\psi_j(X_k^+)\) then gives the
standard EDMD matrix \cite{williams2015data}, which converges to the
Galerkin matrix of \(T\) as the number of transition pairs increases, by
the law of large numbers \cite{korda2018convergence}. For a fixed
\(\widehat T\), BSO instead targets the Galerkin matrix of \(\widehat T\),
whose discrepancy from that of \(T\) is controlled by \(\Delta(V)\). EDMD
with dictionary learning also optimizes the observable space
\cite{li2017extended}. The distinction here is that \(\widehat T\) is
constructed before the space search and can be reused across candidate
spaces. In stochastic systems, its conditional-mean targets also avoid the
conditional-variance term of \cref{rem:stochastic-residual-bias}
\cite{colbrook2024beyond}.

\paragraph{Rank constraints with few transition pairs}
Let \(\mathbf\Psi_X,\mathbf\Psi_{X^+}\in\mathbb R^{\nop\times m}\) be the
data matrices with entries \(\psi_j(X_k)\) and \(\psi_j(X_k^+)\) for a basis
\(\psi_1,\ldots,\psi_m\) of \(V\). The EDMD matrix \(K^{\rm E}\) solves
\(GK^{\rm E}=A^{\rm E}\) with the EDMD Gram and cross matrices
\[
    G\coloneqq\frac{1}{\nop}\mathbf\Psi_X^\top\mathbf\Psi_X,
    \qquad
    A^{\rm E}\coloneqq\frac{1}{\nop}\mathbf\Psi_X^\top\mathbf\Psi_{X^+}.
\]
If \(\nop<m\), then \(G\) is singular, and the minimum-norm solution
\(K^{\rm E}=G^\dagger A^{\rm E}\) satisfies
\(\operatorname{rank}(K^{\rm E})\le\operatorname{rank}(A^{\rm E})\le\nop\);
the same rank bound holds if \(G\) is estimated from additional states.
Hence \(\nop\ge m\) is necessary for an unregularized full-rank EDMD fit.
In BSO, the \(\nop\) transition pairs determine \(\widehat T\), and the
Galerkin matrix \(\widehat K\) is then computed as in
\cref{sec:computational-realization} from points \(x_k\), possibly
weighted, whose construction does not use the states \(X_k^+\).
These points may, for instance, be generated from observed states with the
learned map \(\widehat\Phi\) of a composition estimator, as in
\cref{subsec:small-data-experiment}. With \(n\ge m\) such points and a
nonsingular empirical Gram matrix, \(\widehat K\) is defined whether or not
\(\nop<m\), and it is not subject to this rank bound. This does not imply
that \(\nop\) transitions identify an arbitrary \(m\times m\) operator. The
additional directions are supplied by \(\widehat T\), and the error of
\(\widehat K\) is governed by the learning error and the sampling error
(\cref{subsec:sampling}). See \cite{colbrook2024another} for the regime of
fewer snapshots than dictionary functions within residual DMD.

\begin{remark}[Bias and variance]\label{rem:bias-variance}
The transition data also enter the two errors differently. Let
\(K_T\coloneqq[\ip{\chi_i}{T\chi_j}]_{i,j=1}^m\) be the Galerkin matrix of
\(T\) in a \(\rho\)-orthonormal basis \(\chi\) of \(V\). For independent
transition pairs with \(X_k\sim\rho\), the EDMD cross matrix \(A^{\rm E}\)
formed with \(\chi\) is an unbiased estimator of \(K_T\) by the tower
property, with variance of order \(\nop^{-1}\) under a moment condition.
Conditionally on \(\widehat T\), with sample points that are independent of
it, the corresponding BSO cross matrix
\([\ip{\chi_i}{\widehat T\chi_j}_n]_{i,j=1}^m\) has instead the bias
\([\langle\chi_i,(\widehat T-T)\chi_j\rangle]_{i,j=1}^m\), of norm at most
\(\Delta(V)\), and its only fluctuation is the error of the empirical inner
products (\cref{subsec:sampling}). The learning error thus plays for BSO the
role that the sampling variance plays for EDMD. A target that blends the
observed values, with weight \(1-\beta\), and the values of the
estimator, with weight \(\beta\in[0,1]\), as in \cref{subsec:sst-experiment},
interpolates between the two. With the states \(X_k\) as sample points and
transition pairs that are independent of \(\widehat T\), the bias has norm at
most \(\beta\Delta(V)\), and the fluctuations of the EDMD and BSO cross
matrices are scaled by \(1-\beta\) and \(\beta\).
\end{remark}

\section{Error analysis for spectral approximation and forecasting}
\label{sec:variational}

Throughout this section, \(V\subset\Fn\) is an \(m\)-dimensional subspace
with the computational basis
\(\widehat\Psi=(\widehat\psi_1,\ldots,\widehat\psi_m)\) of
\cref{sec:computational-realization}, which is orthonormal in
\(\ip{\cdot}{\cdot}_n\); \(\widehat\cK=\widehat P_V\widehat T|_V\) is the
computed Galerkin operator \eqref{eq:computed-galerkin-operator}, and
\(\widehat K\) is its matrix. The computational basis is in general not
\(\rho\)-orthonormal; \(\rho\)-orthonormal bases of \(V\) are denoted by
\(\chi\).

The analysis compares \(\widehat\cK\) with the solution operator \(T\) and
is organized by the origin of the errors. Three quantities enter the
bounds. The estimator enters only through its learning error \(\Delta(V)\) of
\eqref{eq:learning-error}, the error of the first stage on the unit ball of
\(V\); no assumption is made on its construction or training. The learning
error is bounded by the errors of \(\widehat T\) on the functions of a
basis of \(V\) (\cref{lem:basis-linear-surrogate}) and by the prediction
error on held-out data (\cref{app:fno-setting}), and it is bounded
explicitly for an FNO in \cref{subsec:empirical-fno}. The empirical inner
products enter through the sampling error \(\eta_V\) defined in
\cref{subsec:sampling}. The space enters through its invariance defect
\(\cI(V)\) and its reconstruction error \(\cR(V)\)
(\cref{sec:galerkin-method}).

The section is organized as follows. \Cref{subsec:sampling} defines the
sampling error and gives the basic estimate: one application of
\(\widehat\cK\) differs on \(V\) from one application of \(T\) by at most
\((\cI(V)+\Delta(V))/(1-\eta_V)\). This estimate holds for every space and
underlies the two main results, the spectral bound of
\cref{subsec:galerkin-spectral-convergence}, which requires in addition
that \(V\) be close to a \(T\)-invariant subspace \(E_\star\), and the
forecast bound of \cref{subsec:forecast-estimates}. For a space selected
by the search, \cref{subsec:dictionary-observable-accuracy} bounds the
objective \(\cJ\), and thereby \(\cI\) and \(\cR\), in terms of the
learning and sampling errors, the best candidate, and the tolerance of the
search. \Cref{thm:quantitative-identification} relates \(\cI(V)\) to the
distance from \(V\) to \(E_\star\), and \cref{subsec:empirical-fno}
inserts an explicit learning rate for FNOs. Proofs are given in
\cref{app:proofs}, which follows the auxiliary lemmas of
\cref{app:conventions,app:empirical-concentration}; supplementary results
are collected in \cref{app:supplements}.

\subsection{One-step error}
\label{subsec:sampling}

The Galerkin matrix \(\widehat K\) is assembled from the sample points
through the empirical inner product \(\ip{\cdot}{\cdot}_n\), which differs
from the inner product of \(L^2(\rho)\). The sampling error measures this
difference on the functions that enter the computation. Let \(W\subset\Fn\)
be a finite-dimensional subspace on which \(\|\cdot\|\) is a norm, that is,
which contains no nonzero function that vanishes \(\rho\)-almost everywhere;
this is assumed of every space to which \(\eta\) is applied below. The
sampling error of the empirical inner product on \(W\) is
\begin{equation}\label{eq:sampling-error}
    \eta(W)
    \coloneqq
    \sup\bigl\{
        |\ip{u}{v}_n-\ip{u}{v}|
        :
        u,v\in W,\ \|u\|=\|v\|=1
    \bigr\}
    =
    \|\widehat G_W-I\|_2,
\end{equation}
where
\(\widehat G_W\coloneqq[\langle\chi^W_i,\chi^W_j\rangle_n]_{i,j=1}^{\dim W}\)
is
the empirical Gram matrix of a \(\rho\)-orthonormal basis \(\chi^W\) of
\(W\); the equality holds because
\(\ip{u}{v}_n-\ip{u}{v}=a^\top(\widehat G_W-I)b\) for \(u,v\in W\) with
coefficient vectors \(a\) and \(b\) in this basis. For the space \(V\), the
functions that enter the computation are those of \(V\), their images under
the estimator, and the reconstruction observables. Accordingly, set
\[
    W_V\coloneqq V+\widehat TV+\spanop\cG,
    \qquad
    \eta_V\coloneqq\eta(W_V),
    \qquad
    \kappa(\eta)\coloneqq\Bigl(\frac{1+\eta}{1-\eta}\Bigr)^{1/2}.
\]
We assume \(\eta_V<1\), which requires \(n\ge\dim W_V\), where
\(\dim W_V\le2m+q\). For a fixed space with a bounded
\(\rho\)-orthonormal basis, the sampling error is of order \(n^{-1/2}\) with
high probability, for independent samples by the matrix Bernstein
inequality \cite{tropp2012user,cohen2013stability,cohen2017optimal} and
for ergodic trajectories with a spectral or pseudo-spectral gap by
\cite{paulin2015concentration} (\cref{app:empirical-concentration}).

\begin{proposition}[One-step error bound]\label{prop:one-step}
Let \(\eta\coloneqq\eta_V<1\). Then the computed Galerkin operator
satisfies
\begin{equation}\label{eq:compression-error}
    \bigl\|\widehat\cK-P_VT|_V\bigr\|
    \le
    \frac{\Delta(V)+\eta\,\cI(V)}{1-\eta},
\end{equation}
and
\begin{equation}\label{eq:one-step-error}
    \|(T-\widehat\cK)v\|
    \le
    e(V)\|v\|,
    \quad
    v\in V,
    \qquad\text{where}\quad
    e(V)\coloneqq\frac{\cI(V)+\Delta(V)}{1-\eta}.
\end{equation}
\end{proposition}

The proposition holds for every space and every estimator and does not
bound \(\cI(V)\) and \(\Delta(V)\); the invariance defect of a selected
space is bounded in \cref{thm:main} and the learning error of an FNO in
\cref{cor:fno-rate}. The estimate decomposes the one-step error into two
parts: the invariance defect \(\cI(V)\), by which the exact Galerkin
compression \(P_VT|_V\) differs from \(T\) on \(V\), and the learning
error, which enters through \eqref{eq:compression-error}. The sampling
error only inflates these terms, by a factor that tends to one as
\(\eta\to0\), because the empirical projection \(\widehat P_V\) is exact
on \(V\) (\cref{lem:sampling}). The spectral bound of \cref{thm:galerkin-spectral}
uses \eqref{eq:compression-error}, and the forecast bound of
\cref{thm:forecast} uses \eqref{eq:one-step-error}.

\paragraph{Data reuse}
The deterministic bounds of this section hold for any sample points,
whether or not these coincide with the states at which the data for
\(\widehat T\) were observed; independence matters only for the
probabilistic bounds on the sampling error. The bounds
stated above for a fixed space apply to \(\eta_V\) when the sample points
are independent of the data used for \(\widehat T\) and for the selection of
\(V\). Otherwise the sampling error has to be controlled uniformly over all
spaces that can occur. When the space is selected on the same sample, this
means uniform control over the candidate family, which is the quantity
\(\eta_\Theta\) of \cref{thm:main}; when the estimator is nonlinear, it
also means uniform control over the spaces generated by the computational
bases (\cref{rem:nonlinear-estimator}). The case in which this uniform
control reduces to a bound for a fixed space is described in
\cref{app:empirical-concentration}; for neural dictionaries with trainable
hidden weights, the uniform control is a hypothesis that is not verified
here.

\subsection{Selection of the observable space}
\label{subsec:dictionary-observable-accuracy}

When the space is not fixed in advance,
\cref{alg:variational-spectral-learning} selects it variationally, by
minimizing the computed objective \eqref{eq:computed-objective} over the
candidates, all of which are evaluated on the same sample. The computed
objective differs from \(\cJ\) in that \(\widehat T\) replaces \(T\) and
the empirical inner product replaces that of \(L^2(\rho)\); the question is
how good the selected space is when measured by \(\cJ\). Let
\(\widehat\theta\in\Theta\) be computed by the algorithm, so that
\begin{equation}\label{eq:near-minimizer}
    \widehat\cJ(\widehat\theta)
    \le
    \inf_{\theta\in\Theta}\widehat\cJ(\theta)+\zeta,
    \qquad
    \zeta\ge0,
\end{equation}
and write \(\widehat V\coloneqq V_{\widehat\theta}\). Since
\(\widehat\theta\) depends on the data, the learning and sampling errors of
\(\widehat V\) are not known in advance, and we use their largest values
over the candidate family,
\[
    \Delta_\Theta\coloneqq\sup_{\theta\in\Theta}\Delta(V_\theta),
    \qquad
    \eta_\Theta\coloneqq\sup_{\theta\in\Theta}\eta_{V_\theta}.
\]

\begin{theorem}[Quasi-optimality of the variational selection]\label{thm:main}
If \(\eta_\Theta<1\), then
\begin{equation}\label{eq:main-bound}
    \cJ(\widehat V)^{1/2}
    \le
    \varepsilon_\Theta
    \coloneqq
    \kappa(\eta_\Theta)^2
    \Bigl(
    \inf_{\theta\in\Theta}\cJ(V_\theta)^{1/2}
    +2\Delta_\Theta
    +\zeta^{1/2}
    \Bigr).
\end{equation}
\end{theorem}

The selected space is thus nearly as good as the best candidate; with exact
inner products, it is as good up to \(2\Delta_\Theta+\zeta^{1/2}\). The
bound \eqref{eq:main-bound} separates the four error sources of the
selection step, namely the approximation capacity of the candidate family,
\(\inf_{\theta}\cJ(V_\theta)\), the learning error \(\Delta_\Theta\) of the
estimator, the optimization tolerance \(\zeta\), and the sampling
error, which enters only through the factor
\(\kappa(\eta_\Theta)^2=(1+\eta_\Theta)/(1-\eta_\Theta)\). Both terms of the
objective are thereby controlled,
\(\cI(\widehat V)\le\varepsilon_\Theta\) and
\(\cR(\widehat V)\le\varepsilon_\Theta/\alpha^{1/2}\), and so is the one-step
error \eqref{eq:one-step-error} of the selected space,
\(e(\widehat V)\le(\varepsilon_\Theta+\Delta_\Theta)/(1-\eta_\Theta)\),
since \(\eta_{\widehat V}\le\eta_\Theta\) and
\(\Delta(\widehat V)\le\Delta_\Theta\). The computed objective value also
bounds the exact one a posteriori,
\(\cJ(\widehat V)^{1/2}\le\kappa(\eta_\Theta)\,\widehat\cJ(\widehat\theta)^{1/2}
+\Delta_\Theta\) (\cref{lem:objective-equivalence}).

The objective is Lipschitz continuous with respect to \(\dist\), with
constant \(4M_T^2+2\alpha\max_j\|g_j\|^2\), since \(\cI\) and \(\cR\) are
Lipschitz with constants \(2M_T\) and \(\max_j\|g_j\|\) and bounded by
\(M_T\) and \(\max_j\|g_j\|\). This continuity yields a convergence
statement along increasing levels. Suppose that the candidate spaces
\(V_\theta\), \(\theta\in\Theta_p\), of \cref{sec:galerkin-method} become
dense, as \(p\to\infty\), in a class of \(m\)-dimensional subspaces that
contains them and on which \(\cJ\) attains its infimum. Suppose also that
\(\Delta_{\Theta_p}\), \(\eta_{\Theta_p}\), and \(\zeta\) tend to zero
along the estimators, samples, and tolerances used at level \(p\). Then
\eqref{eq:main-bound} shows that \(\cJ(\widehat V)\) converges to this
infimum for the spaces \(\widehat V\) selected at level \(p\). If the
infimum is attained at a \(T\)-invariant subspace \(E_\star\) and the
reconstruction observables span \(E_\star\), then
\(\dist(\widehat V,E_\star)\to0\) by \cref{lem:localization}, and
\cref{thm:galerkin-spectral} gives the convergence of the computed
eigenvalues and spectral subspaces. The choice of the level \(p\) is
discussed in \cref{rem:level}.

\subsection{Spectral approximation}
\label{subsec:galerkin-spectral-convergence}

Let \(E_\star\) be an \(m\)-dimensional \(T\)-invariant subspace, the target
of the spectral computation, so that the eigenvalues to be computed are
those of \(T|_{E_\star}\). Two questions arise. If \(V\) is close to \(E_\star\), how
accurate is the computed spectrum? And how can a space that is selected by
the objective be guaranteed to be close to \(E_\star\)?

\begin{theorem}[Spectral error bound]\label{thm:galerkin-spectral}
Let \(\eta\coloneqq\eta_V<1\) and \(d\coloneqq\dist(V,E_\star)<1\).
Then there is an isometric isomorphism \(\cQ\colon E_\star\to V\) with
\(\|\cQ-\iota\|_{\cB(E_\star,\cH)}\le2d\), where \(\iota\colon E_\star\to\cH\)
is the inclusion and \(\cQ\) is regarded as a map into \(\cH\), such that
\begin{equation}\label{eq:aligned-matrix-bound}
    \bigl\|\cQ^*\widehat\cK\cQ-T|_{E_\star}\bigr\|
    \le
    \frac{\Delta(V)+4M_T\,d}{1-\eta}.
\end{equation}
Consequently, consider a sequence of \(m\)-dimensional subspaces \(V_k\),
possibly with estimators and samples that depend on \(k\), with computed
Galerkin operators \(\widehat\cK_k\) and matrices \(\widehat K_k\). If
\(\Delta(V_k)\to0\), \(\dist(V_k,E_\star)\to0\), and
\(\sup_k\eta_{V_k}<1\), then the eigenvalues of \(\widehat K_k\) converge to
those of \(T|_{E_\star}\), counted with algebraic multiplicity, and the
spectral subspaces of \(\widehat\cK_k\) converge in \(\dist\) to the
corresponding spectral subspaces of \(T|_{E_\star}\).
\end{theorem}

The operator \(\cQ^*\widehat\cK\cQ=\cQ^{-1}\widehat\cK\cQ\) acts on
\(E_\star\) and has the same eigenvalues as \(\widehat K\), and
\eqref{eq:aligned-matrix-bound} bounds its distance from \(T|_{E_\star}\),
which gives the convergence statement. No
spectral isolation is assumed, and \(E_\star\) may be spanned by
eigenfunctions of eigenvalues embedded in the continuous spectrum, as for
measure-preserving systems with mixed spectrum.

The theorem concerns spaces \(V\) whose dimension equals that of
\(E_\star\), so that all of \(V\) is nearly invariant. In data-driven
spectral computations, \(V\) is commonly larger than the invariant subspace
of interest. The computed eigenpairs that are not associated with this
subspace may then be inaccurate or spurious, and they are discarded, for
instance by means of the residual bound of
\cref{prop:eigenpair-certificate}. This filtering is standard and is not
analyzed here.

We turn to the second question. The distance \(\dist(V,E_\star)\) in
\cref{thm:galerkin-spectral} is not known in practice, whereas the
invariance defect is controlled by the selection step (\cref{thm:main}). The
next theorem shows that the defect and the distance are equivalent near
\(E_\star\), provided that the part of the spectrum that belongs to
\(E_\star\) is separated from the rest, in the following sense. With
respect to \(\cH=E_\star\oplus E_\star^\perp\), the operator \(T\) is block
upper triangular, with diagonal blocks \(A_\star\coloneqq T|_{E_\star}\)
and \(D_\star\coloneqq(I-P_\star)T|_{E_\star^\perp}\) and off-diagonal block
\(B_\star\coloneqq P_\star T|_{E_\star^\perp}\), and we set
\[
    \delta_\star
    \coloneqq
    \inf\bigl\{\|D_\star Z-ZA_\star\|:
    Z\in\cB(E_\star,E_\star^\perp),\ \|Z\|=1\bigr\},
    \qquad
    r_\star\coloneqq\min\Bigl\{1,\frac{\delta_\star}{2\|B_\star\|}\Bigr\},
\]
with \(r_\star\coloneqq1\) if \(B_\star=0\). The number \(\delta_\star\) is
the separation of \(A_\star\) and \(D_\star\)
\cite{stewart1973error,stewart1990matrix}, and \(r_\star\) determines the
size of the neighborhood of \(E_\star\) in which the equivalence holds. The
separation is positive if \(E_\star\) is the spectral subspace of an
isolated cluster (\cref{sec:spec_pert_theory}); for self-adjoint \(T\) and a
cluster formed by the \(m\) largest eigenvalues, it is the spectral gap, and
\(r_\star=1\) (\cref{app:identification}).

\begin{theorem}[Invariance defect and distance]\label{thm:quantitative-identification}
Let \(\delta_\star>0\), and let \(V\) be an \(m\)-dimensional subspace with
\(\dist(V,E_\star)\le r_\star/\sqrt2\). Then
\begin{equation}\label{eq:local-identification}
    \frac{\delta_\star}{4}\,\dist(V,E_\star)
    \le
    \cI(V)
    \le
    2M_T\,\dist(V,E_\star).
\end{equation}
The upper bound holds for every \(m\)-dimensional subspace \(V\) and does
not require \(\delta_\star>0\).
\end{theorem}

The theorem is a variant of Stewart's perturbation theorem for invariant
subspaces \cite{stewart1973error,stewart1990matrix}, in which the
separation is measured at the target subspace \(E_\star\) rather than at
\(V\). The restriction to a neighborhood cannot be removed: a small defect
does not by itself identify a particular invariant subspace as every
invariant subspace has zero defect. The reconstruction term of the
objective brings the selected space into this neighborhood. If the
reconstruction observables resolve \(E_\star\) and lie sufficiently close
to it, and \(\varepsilon_\Theta/\alpha^{1/2}\) is sufficiently small, then
the selected space satisfies
\(\dist(\widehat V,E_\star)\le4\varepsilon_\Theta/\delta_\star\), and
\cref{thm:galerkin-spectral} bounds its spectral error in terms of the
quantities of \cref{thm:main}. The precise conditions and the bias caused by
the reconstruction term are given in \cref{app:identification}
(\cref{cor:identification}).

\subsection{Forecast bounds}
\label{subsec:forecast-estimates}

For an observable \(f\in\spanop\cG\),
\cref{alg:variational-spectral-learning} returns the Galerkin forecast
\[
    f^{(N)}
    \coloneqq
    \widehat\cK^N\widehat P_Vf
    =\sum_{j=1}^m(\widehat K^N\widehat c)_j\widehat\psi_j,
    \qquad
    \widehat c\coloneqq\bigl(\ip{\widehat\psi_j}{f}_n\bigr)_{j=1}^m,
\]
of \(T^Nf=T_{N\tau}f\), the conditional mean of the observable after \(N\)
lag steps. Since \(f\in W_V\), the sampling error of the projection is
covered by \(\eta_V\). Let
\(M\coloneqq\max\{M_T,\|\widehat\cK\|\}\), where \(\|\widehat\cK\|\) is the
operator norm of \(\widehat\cK\) on \(V\) with the \(L^2(\rho)\) norm.

\begin{theorem}[Forecast error bound]\label{thm:forecast}
Let \(\eta\coloneqq\eta_V<1\), and let \(e(V)\) be the one-step error of
\cref{prop:one-step}. Then, for every \(N\ge1\),
\begin{equation}\label{eq:matrix-power-forecast-bound}
    \|T^Nf-f^{(N)}\|
    \le
    \frac{M^N}{1-\eta}\,\|(I-P_V)f\|
    +NM^{N-1}e(V)\,\|P_Vf\|.
\end{equation}
If, moreover, \(V=\widehat V\) is the space selected in \cref{thm:main},
\(\eta_\Theta<1\), and \(f\in\cG\), then
\begin{equation}\label{eq:selected-forecast-bound}
    \|T^Nf-f^{(N)}\|
    \le
    \frac{1}{1-\eta_\Theta}
    \Bigl(
    M^N\,\frac{\varepsilon_\Theta}{\alpha^{1/2}}
    +NM^{N-1}\bigl(\varepsilon_\Theta+\Delta_\Theta\bigr)\|f\|
    \Bigr).
\end{equation}
\end{theorem}

For \(N=1\) and \(f\in V\), the bound
\eqref{eq:matrix-power-forecast-bound} is the one-step error
\eqref{eq:one-step-error}; the general case follows by applying
\cref{prop:one-step} at each of the \(N\) steps. The first term accounts for
the part of \(f\) that is not represented in \(V\), and the second
accumulates the one-step error over \(N\) steps; if \(M\le1\), the error
grows at most linearly in \(N\). This bound involves only the space in use
and requires neither the selection result of
\cref{subsec:dictionary-observable-accuracy} nor the identification of an
invariant subspace. For the selected space,
\eqref{eq:selected-forecast-bound} expresses the forecast error through
\(M\) and the quantities of \cref{thm:main}.

If \(\rho\) is invariant and the computed Galerkin operator preserves
constants, leaves the mean-zero subspace of \(V\) invariant, and is stable
on it, the factors \(M^N\) and
\(NM^{N-1}\) are replaced by constants, so that the bound is uniform in
\(N\) (\cref{cor:uniform-forecast}). This result, a form whose constraints
can be imposed from the sample, and the effect of the corresponding
modifications of \(\widehat\cK\) are given in
\cref{app:forecast-supplements}.

\subsection{A learning rate for Fourier neural operators}
\label{subsec:empirical-fno}

The preceding results involve the accuracy of the first stage only through
the learning error \(\Delta\), and any guarantee that controls \(\Delta\)
can be inserted. The learning error \(\Delta\) can also be bounded from
above by the prediction error on held-out data (\cref{app:fno-setting}). For deterministic
dynamics, consider the composition estimator
\(\widehat Tf=f\circ\widehat\Phi\) of \cref{rem:neural-operator}, in which
\(\widehat\Phi\) is a learned approximation of the map \(\Phi\) that
advances the state by one lag, so that
\((\widehat T-T)v=v\circ\widehat\Phi-v\circ\Phi\). Let \(\dist_\Omega\) be a
metric on \(\Omega=Y\), and let \(\Lip(V)\) denote the largest Lipschitz
constant of a function \(v\in V\) with \(\|v\|=1\), on a fixed set
\(\Omega_0\) that contains the values of \(\Phi\) and \(\widehat\Phi\). Then,
as shown in \cref{app:fno-setting},
\begin{equation}\label{eq:composition-learning-error}
    \Delta(V)
    \le
    \Lip(V)\,
    \bigl\|\dist_\Omega(\widehat\Phi,\Phi)\bigr\|_{L^2(\rho)},
\end{equation}
where \(\dist_\Omega(\widehat\Phi,\Phi)\) denotes the function
\(y\mapsto\dist_\Omega(\widehat\Phi(y),\Phi(y))\).
The first factor depends only on \(V\) and \(\Omega_0\). It converts an
error of the state into an error of the observables, and it grows with the
richness of the space. For polynomial dictionaries of degree \(p\) on a
bounded interval on which \(\rho\) has a density bounded below, it is at
most of order \(p^3\) (\cref{rem:level}). The second factor is the error
of the learned map and does not depend on \(V\); for an FNO, it is bounded
in \cite{chandramoorthy2026spectral}. In that setting, \(\Phi\) is the flow
map over a fixed time of a dissipative evolution equation with polynomial
nonlinearity on the \(d_x\)-dimensional torus \(\mathbb T^{d_x}\), the
inputs have Sobolev regularity \(a>d_x/2\), and \(\dist_\Omega\) is the
\(L^2\) distance. The estimator \(\widehat\Phi\) is an empirical risk
minimizer over an FNO class, trained on \(\nop\) exact independent pairs
\((u_i,\Phi(u_i))\) with \(u_i\sim\rho\). By
\cite[Theorem~3.6]{chandramoorthy2026spectral}, there are constants
\(C_0>0\) and \(n_0\) such that, for \(\nop\ge n_0\), with probability at
least \(1-\nop^{-1}\),
\[
    \bigl\|\dist_\Omega(\widehat\Phi,\Phi)\bigr\|_{L^2(\rho)}^2
    \le
    C_0\,\nop^{-1/(6+2d_x/a)} .
\]
The precise assumptions are stated in \cref{app:fno-setting}. Inserting
this rate in \eqref{eq:composition-learning-error} gives the following bound.

\begin{corollary}[Learning error for an FNO]\label{cor:fno-rate}
In the setting of \cref{app:fno-setting}, for \(\nop\ge n_0\), with
probability at least \(1-\nop^{-1}\),
\[
    \Delta(V)\le\Lip(V)\,C_0^{1/2}\,\nop^{-1/(12+4d_x/a)}
\]
simultaneously for all \(m\)-dimensional subspaces \(V\subset\Fn\). In
particular, the learning error in \cref{thm:main} satisfies
\[
    \Delta_\Theta\le\Lip_\Theta\,C_0^{1/2}\,\nop^{-1/(12+4d_x/a)},
    \qquad
    \Lip_\Theta\coloneqq\sup_{\theta\in\Theta}\Lip(V_\theta).
\]
\end{corollary}

With this bound, the learning error in the one-step, spectral, and forecast
estimates of this section is expressed through the number \(\nop\) of
training pairs.
The cited rate concerns exact independent pairs and an exact empirical risk
minimizer; it does not cover dependent trajectory data, noisy pairs, or the
optimization error of training.

\section{Numerical experiments}\label{sec:numerical-experiments}
This section applies the BSO method to two spectral benchmarks, the
Ornstein--Uhlenbeck and triple-well diffusions, and to four forecasting
problems involving Lorenz--63 \cite{lorenz1963deterministic}, Lorenz--96
\cite{lorenz1995predictability}, global sea-surface temperature (SST), and
resting-state electroencephalography (EEG). A further experiment treats the
case in which the number of observed transition pairs is smaller than the
dimension of the Galerkin matrix (\cref{subsec:small-data-experiment}).
Throughout this section, \emph{BSO} denotes the two-stage method of
\cref{sec:methodology} with an FNO \cite{li2021fourier} as the operator
estimator. In the spectral examples, \(\widehat T\) approximates \(T_\tau\)
directly and is nonlinear in the observable
(\cref{rem:nonlinear-estimator}), whereas in the forecasting examples it is
the composition estimator of \cref{rem:neural-operator}. The experiments
are not designed to test the bounds of \cref{sec:variational}
quantitatively, and several lie outside their hypotheses. Implementation
details are collected in \cref{app:numerical-details}.

\subsection{Spectral approximation}\label{subsec:spectral-experiments}

\subsubsection{Ornstein--Uhlenbeck process}
\label{subsubsec:ou-experiment}

We consider
\[
    \dd X_t=-X_t\dd t+\dd W_t,
    \qquad
    \cL f=-xf'+\tfrac12f''.
\]
This benchmark has an analytic spectrum, against which the computed rates
and modes can be checked. At lag \(\tau=0.15\), BSO learns the fixed-lag
operator and a 20-dimensional Galerkin space without using analytic
eigenvalues or eigenfunctions. The training pairs \((f,T_\tau f)\) are
obtained by evolving the observables \(f\) over one lag with a numerical
discretization of the backward equation (\cref{app:numerical-details}).
\Cref{fig:ou-learned} compares the recovered rates with the exact generator
spectrum \(\lambda_n=-n\), \(n=0,1,2,\ldots\), and the recovered modes with
the Hermite eigenfunctions \cite{pavliotis2014stochastic}. 

\begin{figure}[tbp]
\centering
\includegraphics[width=0.60\linewidth]{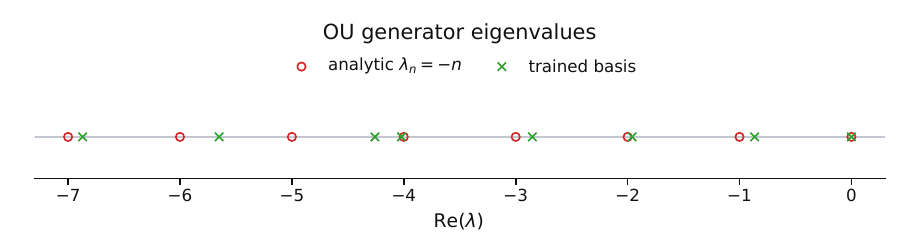}
\par\medskip
\includegraphics[width=0.60\linewidth]{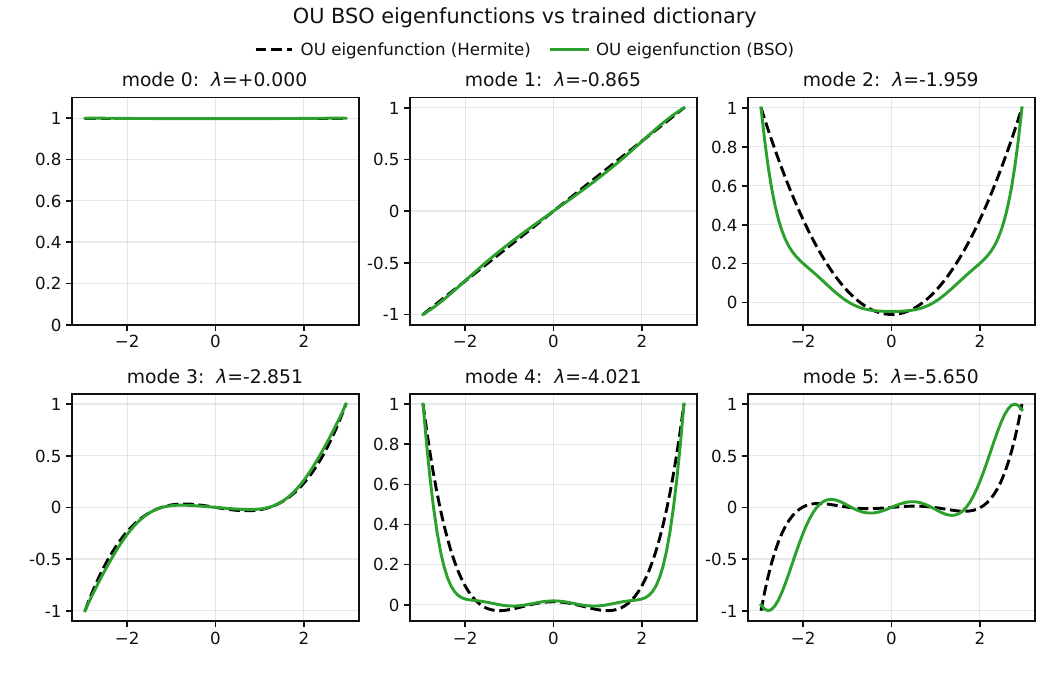}
\caption{OU spectral approximation at \(\tau=0.15\). Top:
analytic generator rates and the BSO (learned dictionary) estimates, labeled
``trained basis'' in the legend, each matched to the nearest analytic rate.
Bottom: postprocessed BSO (learned dictionary) eigenfunctions and the
analytic Hermite modes on \(|x|\le3\).}
\label{fig:ou-learned}
\end{figure}

\subsubsection{Triple-well diffusion}
\label{subsubsec:triple-well-experiment}

The second benchmark is the triple-well reflected diffusion of
\cite{xu2025koopmansemigroup},
\[
    \dd X_t=-\nabla U(X_t)\dd t+1.09\dd W_t,
\]
with the three-hole potential \(U\) of \cite{metzner2006illustration}
specified in \cref{app:numerical-details}.
The operator is learned at the lag \(\tau=0.30\), and the multipliers in
\cref{fig:triple-well-spectrum} are those at the lag \(3\tau=0.9\). A
numerical reversibility constraint is imposed by symmetrizing the reduced
matrix; this does not assert exact self-adjointness of the uniform-grid
projection, and its effect on the error bounds is described in
\cref{rem:postprocessing}. The experiment concerns the three leading
eigenpairs, for which \cref{fig:triple-well-spectrum} shows the multipliers
and rates together with the left--right and upper--lower metastable
partitions. The Galerkin space is larger than the invariant subspace
spanned by these modes, so the remaining computed eigenvalues are not
expected to be accurate (\cref{subsec:galerkin-spectral-convergence}).

\begin{figure}[tbp]
\centering
\includegraphics[width=0.60\linewidth]{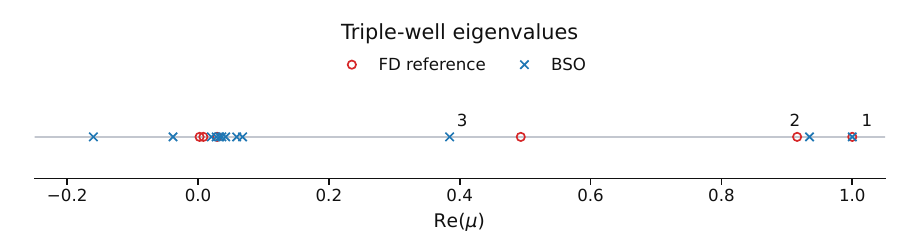}
\par\medskip
\includegraphics[width=0.60\linewidth]{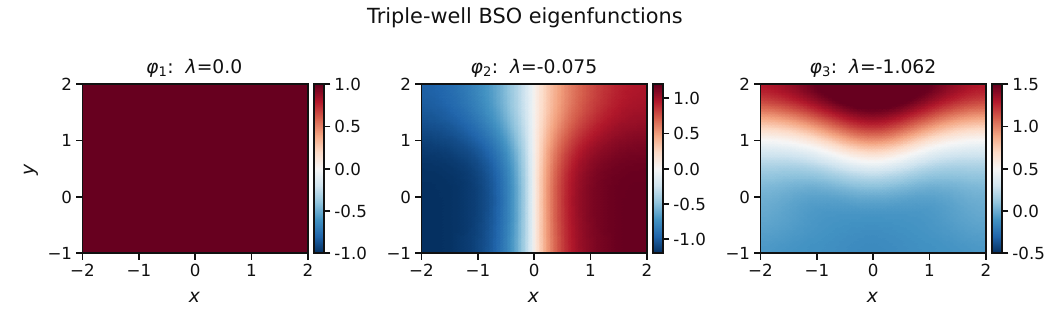}
\caption{Triple-well spectral approximation at the lag \(3\tau=0.9\). Top:
leading BSO (learned dictionary) multipliers and finer-grid
finite-difference references. Bottom: the stationary, left--right, and
upper--lower modes, paired by one-to-one matching; the panel labels give
the continuous-time rates \(\lambda=(3\tau)^{-1}\log\mu\) corresponding to
the multipliers \(\mu\) of the top panel (\cref{sec:spec_pert_theory}).}
\label{fig:triple-well-spectrum}
\end{figure}

% Keep the subsection heading on the same page as its first subsubsection.
\needspace{4\baselineskip}
\subsection{Forecasting}\label{subsec:forecasting}
\subsubsection{Chaotic systems}
\label{subsec:lorenz-experiments}

We compare autonomous BSO rollouts with autoregressive rollouts of the
same fixed FNO. Lorenz--63 \cite{lorenz1963deterministic} and Lorenz--96 \cite{lorenz1995predictability} are given by
\[
    \dot x=10(y-x),\qquad
    \dot y=x(28-z)-y,\qquad
    \dot z=xy-\tfrac83z,
\]
and
\[
    \dot x_i=(x_{i+1}-x_{i-2})x_{i-1}-x_i+8,
    \qquad i=1,\ldots,20,
\]
with periodic indices. In both systems the FNO acts on a smooth curve
sampled at 20 sensor sites that encodes the state. Let \(y_{o,j,k}\) and
\(\widehat y_{o,j,k}\)
denote the observed and predicted values at held-out origin
\(o\in\{1,\ldots,64\}\), sensor site \(j\in\{1,\ldots,20\}\), and lead
\(k\). Forecast errors are pooled over origins and sites,
\[
    E_k^{\rm Lor}
    \coloneqq
    \frac{\bigl(\sum_{o,j}
    |\widehat y_{o,j,k}-y_{o,j,k}|^2\bigr)^{1/2}}
    {\bigl(\sum_{o,j}|y_{o,j,k}|^2\bigr)^{1/2}}.
\]
For Lorenz--63 (\cref{fig:lorenz-forecast}), the two methods have similar
errors at short leads: the autoregressive FNO has lower error for leads up
to about \(t=1\), and BSO has lower error thereafter. For Lorenz--96, BSO
has lower error throughout the positive forecast horizon. At long lead
times, the forecast errors of both methods remain bounded over the tested
horizon.

\begin{figure}[tbp]
\centering
\includegraphics[width=0.60\linewidth]{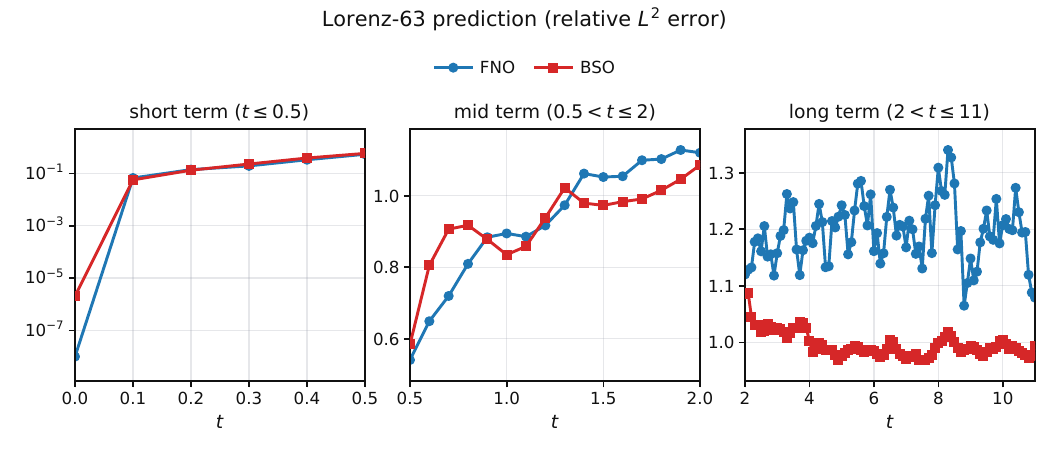}
\par\medskip
\includegraphics[width=0.60\linewidth]{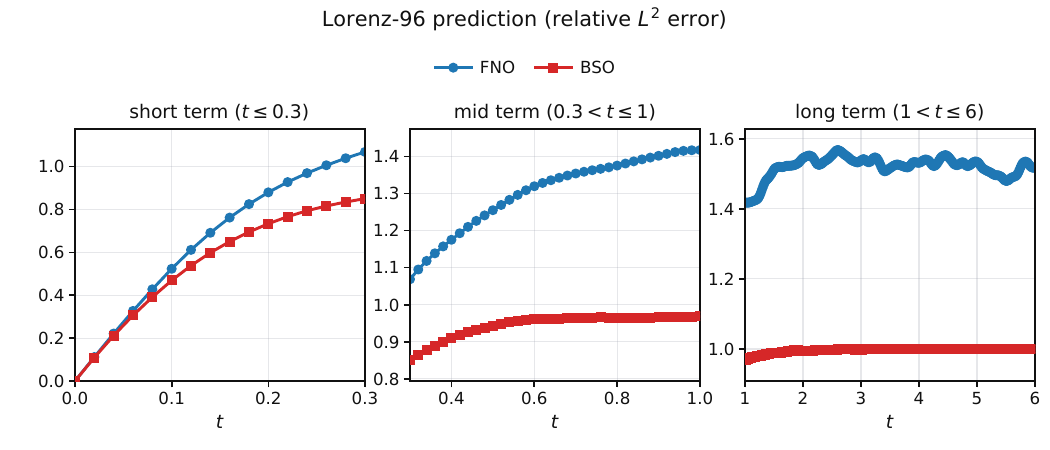}
\caption{Pooled relative \(L^2\) forecast error of BSO (fixed dictionary)
and the autoregressive FNO over 64 held-out origins for Lorenz--63 (top) and
Lorenz--96 (bottom). The short-term
Lorenz--63 panel uses a logarithmic vertical axis; the remaining panels use
linear axes with different vertical scales.}
\label{fig:lorenz-forecast}
\end{figure}

\subsubsection{Global sea-surface temperature}
\label{subsec:sst-experiment}

The SST experiment uses global fields on a \(32\times64\)
latitude--longitude grid from 1993--2022. SST and the surface wind stress are derived from the ERA5 reanalysis
\cite{hersbach2020era5} as distributed in the WeatherBench~2 archive \cite{rasp2024weatherbench2}, while the remaining ocean fields are taken from the NCEP Global Ocean Data Assimilation System (GODAS) \cite{behringer2004godas}. The 11 input fields are listed in \cref{tab:sst-fields} of \cref{app:numerical-details}. The model is selected using 1993--2016 for fitting and 2017--2019 for chronological validation, then refitted on 1993--2019 and evaluated on 2020--2022. All 11 fields initialize the learned models, but only SST is verified.

In this experiment the reduced matrix is not the plain Galerkin compression
of \cref{sec:methodology}. Instead, it is fitted by ridge regression to a
blended one-day target, which for a weight \(\beta=0.4\) is \(1-\beta\)
times the observed next-day coefficient vector plus \(\beta\) times the
coefficient vector of the FNO prediction. This target interpolates between
EDMD (\(\beta=0\)) and the BSO compression (\(\beta=1\)) compared in
\cref{subsec:smalldata}; \cref{rem:bias-variance} describes the bias and
the fluctuation of the cross matrix of such a blended target, for
independent data and without regularization. The fitted matrix is then
subjected to stability adjustments (\cref{rem:postprocessing}); the
settings are given in \cref{app:numerical-details}.

Forecast skill (\cref{fig:sst-results}, top) is evaluated through day 90
over 96 held-out origins using the area-weighted root-mean-square error
(RMSE), the mean absolute error (MAE), and the centered anomaly correlation
coefficient (ACC). The baselines are persistence, calendar-day climatology,
and damped persistence,
\[
    \operatorname{SST}_{\rm DP}(t+\ell)
    \coloneqq
    \operatorname{SST}_{\rm clim}(t+\ell)+0.985^\ell
    \bigl(\operatorname{SST}(t)-\operatorname{SST}_{\rm clim}(t)\bigr),
\]
where \(\ell\) is the lead in days, \(\operatorname{SST}_{\rm clim}\)
denotes the calendar-day climatology, and the decay factor is selected from
the fitting period. BSO is competitive with damped persistence at short
leads and retains anomaly skill over the longer evaluated horizon, as the
ACC panel of \cref{fig:sst-results} shows.

A separate 540-day test (\cref{fig:sst-results}, bottom) assesses numerical
stability rather than forecast skill. Over 48 origins, the BSO rollout
remains finite through day 540 with SST RMSE \(0.66\) K, the FNO reaches
\(816\) K, and the spherical Fourier neural operator
(SFNO)~\cite{bonev2023spherical} first becomes nonfinite at day 171. No
clipping or replacement is used. These values do not rank
seasonal-to-annual forecast skill; they show that the BSO rollout remains
numerically stable over the test period, whereas the autoregressive
rollouts of the neural operators do not.

\begin{figure}[tbp]
\centering
\includegraphics[width=0.60\linewidth]{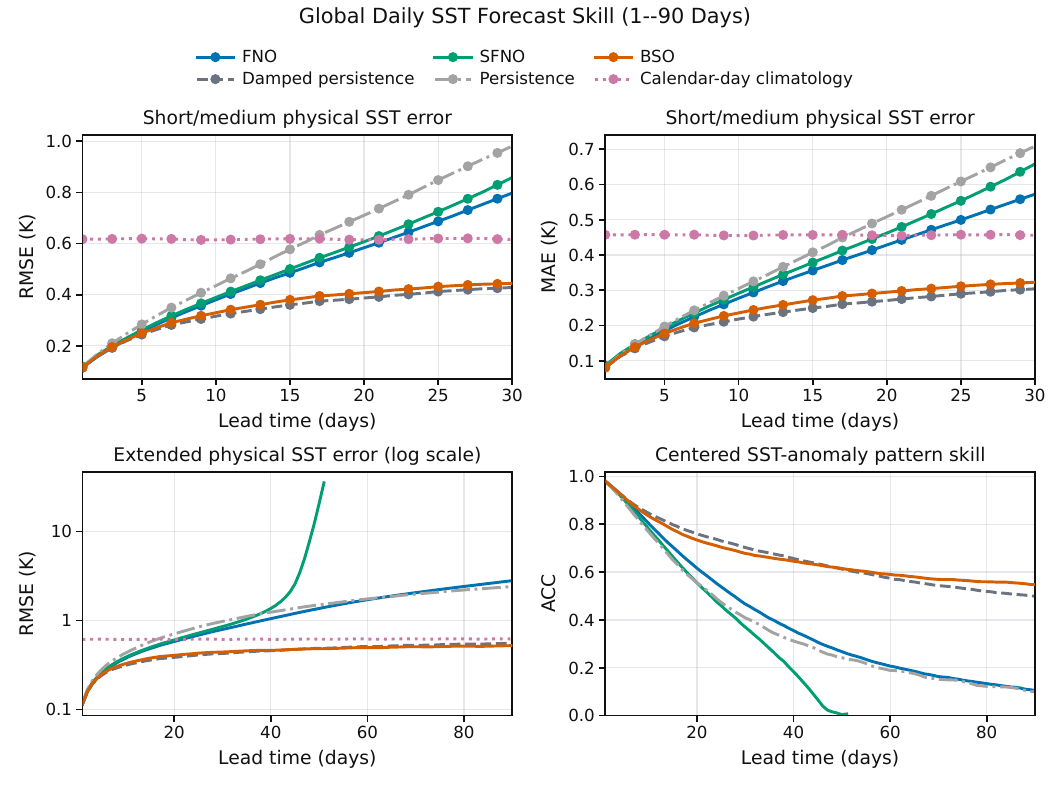}
\par\medskip
\includegraphics[width=0.60\linewidth]{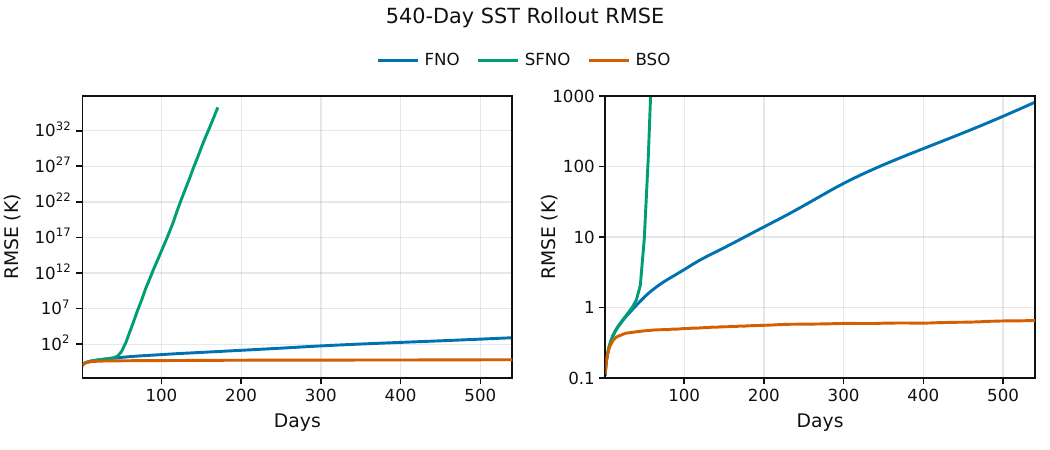}
\caption{Global SST evaluation. Top: forecast skill of BSO (fixed
dictionary), FNO, SFNO, and the three baselines over 96
held-out 2020--2022 origins, showing area-weighted RMSE and MAE through day 30 and RMSE and centered anomaly ACC through day 90. The SFNO curve ends at the first lead time for which scores are not available for all 96 origins. Bottom: area-weighted SST RMSE in the 540-day numerical-stability test over 48 origins, shown over the full and a zoomed logarithmic range.}
\label{fig:sst-results}
\end{figure}

\subsubsection{Resting-state EEG rollout}
\label{subsec:eeg-experiment}

The final experiment uses subject S029, run R02, from the PhysioNet EEG Motor Movement/Imagery Dataset
\cite{schalk2009eegmmidb,schalk2004bci2000,pollard2026physionet,goldberger2000physionet}. Evaluation consists of one autonomous 40-step rollout from 15 to 19 seconds, with no observations supplied after
launch. With \(y_\ell\) the observed and \(\widehat y_\ell\) the predicted
channel vectors at step \(\ell\), the cumulative relative error through
step \(k\) is
\[
    E_k^{\rm EEG}
    \coloneqq
    \frac{\left(\sum_{\ell=1}^k\|\widehat y_\ell-y_\ell\|_2^2\right)^{1/2}}
         {\left(\sum_{\ell=1}^k\|y_\ell\|_2^2\right)^{1/2}}.
\]
The comparators are direct autoregressive FNO, EDMD with dictionary learning (EDMD-DL) \cite{li2017extended}, Gaussian-kernel EDMD \cite{williams2015kernel}, and an affine vector-autoregressive model in principal component analysis (PCA) coordinates (PCA--VAR). The direct FNO and the FNO used inside BSO have different fixed hyperparameters. This comparison therefore does not isolate the effect of the Galerkin step. The reduced matrix of BSO is ridge-regularized (\cref{app:numerical-details}); the effect of such a modification on the bounds is described in \cref{rem:postprocessing}.

BSO has the lowest cumulative error through step 40, \(E_{40}^{\rm EEG}\)
(\cref{tab:eeg-result,fig:eeg-rollout}), which is \(14.2\%\) below that of the direct FNO. This is
a descriptive result for one subject, one run, and one forecast window;
resting EEG is only approximately autonomous and no uncertainty interval is
available.

\begin{table}[tbp]
\centering
\caption{EEG experiment: relative \(L^2\) error accumulated through step 40
of the four-second rollout, \(E_{40}^{\rm EEG}\), and the relative reduction
achieved by BSO with respect to each comparator, computed from unrounded
errors.}
\label{tab:eeg-result}
\begin{tabular}{lcc}
\hline
Method & Error \(E_{40}^{\rm EEG}\) & Reduction with BSO \\
\hline
BSO & 0.749 & -- \\
Direct FNO & 0.874 & \(14.2\%\) \\
EDMD-DL & 0.941 & \(20.4\%\) \\
Kernel EDMD & 0.945 & \(20.7\%\) \\
PCA--VAR & 0.963 & \(22.2\%\) \\
\hline
\end{tabular}
\end{table}

\begin{figure}[tbp]
\centering
\includegraphics[width=0.60\linewidth]{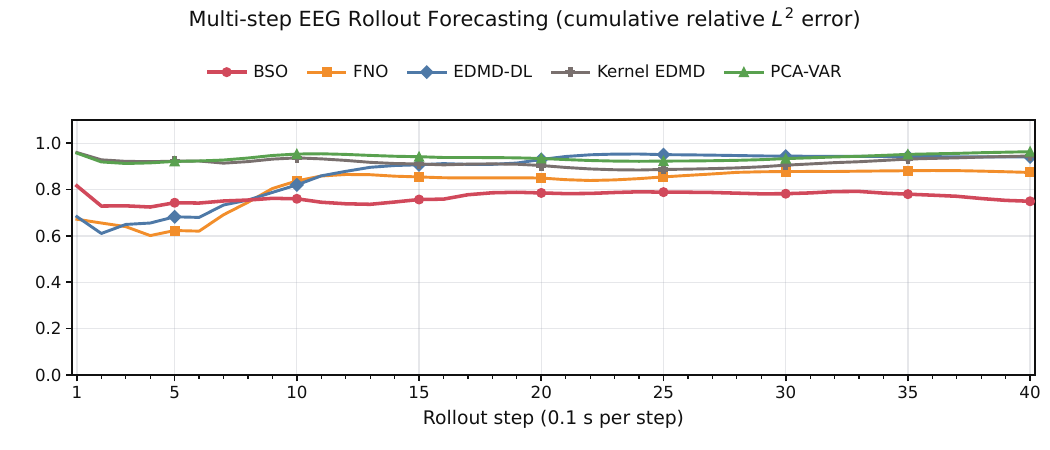}
\caption{Cumulative relative \(L^2\) error of BSO (learned dictionary) and
the comparators for the four-second autonomous rollout on the S029R02 resting
EEG record. Each step advances 0.1 seconds.}
\label{fig:eeg-rollout}
\end{figure}

\subsection{Small-data comparison between BSO and rank-deficient EDMD}\label{subsec:small-data-experiment}
This experiment illustrates the finite-sample distinction described in
\cref{subsec:smalldata}. The system is Lorenz--96 with 40 sites, and the
dictionary is linear in the state, so that \(m=40\); the configuration
therefore differs from the 20-site system and nonlinear dictionary of
\cref{subsec:lorenz-experiments}.

The BSO Galerkin matrix is constructed as follows. A neural operator is
fitted to the \(\nop\) transition pairs and then held fixed. Weighted nodes
\((x_k,w_k)\), \(k=1,\ldots,n\), of the form used in
\cref{sec:computational-realization} are built from states \(x_k\) generated
by applying the fixed operator to the observed states, with weights \(w_k\)
summing to one. These generated states serve only as integration nodes for
the Galerkin inner product. With the data matrices \(\mathbf\Psi\) and
\(\mathbf\Psi_{\widehat T}\) of \eqref{eq:bso-least-squares} evaluated at
the nodes \(x_k\), and with \(w\coloneqq(w_1,\ldots,w_n)\), the matrix
is the weighted least-squares fit
\[
    \widehat K
    \coloneqq
    \bigl(\mathbf\Psi^\top\operatorname{diag}(w)\mathbf\Psi\bigr)^{-1}
    \mathbf\Psi^\top\operatorname{diag}(w)\mathbf\Psi_{\widehat T},
\]
which reduces to \eqref{eq:bso-least-squares} for equal weights. It is the
matrix, in the basis \(\psi_1,\ldots,\psi_m\), of the computed Galerkin
operator of \cref{sec:computational-realization} for these weights. The
nodes contain no observed states
\(X_k^+\), so the Galerkin matrix is defined for every \(\nop\) for which
\(\mathbf\Psi^\top\operatorname{diag}(w)\mathbf\Psi\) is nonsingular. The
EDMD baselines are computed in the same dictionary from the same \(\nop\)
observed pairs, and the one-step error of the neural operator itself is
included as a reference for the accuracy of the estimator before Galerkin
compression.

In \cref{fig:smalldata}, for \(\nop<m\), BSO has a lower one-step prediction
error on held-out transition pairs, measured in the relative \(L^2\) norm,
than the EDMD variants. The ridge-regularized and minimum-norm
\(m\)-dimensional fits have high one-step error, and the rank-restricted
baseline reaches the accuracy of the full fit only once \(\nop\) exceeds
\(m\). The one-step predictions of BSO differ from those of the neural
operator from which it is projected by at most \(0.024\) in relative
\(L^2\) norm throughout. A full analysis of this regime is left to future
work.

\begin{figure}[tbp]
    \centering
    \includegraphics[width=0.45\linewidth]{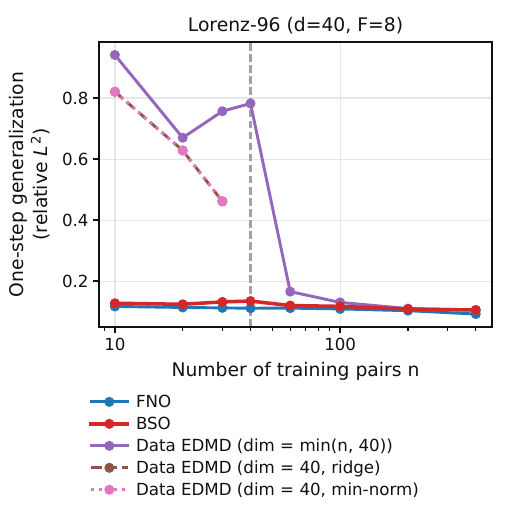}
    \caption{One-step prediction error on held-out transition pairs
    (relative \(L^2\)) versus the number \(\nop\) of observed transition
    pairs for Lorenz--96 with 40 sites and \(m=40\) (vertical dashed line).
    BSO (fixed dictionary) is available for every \(\nop\) shown. The dashed and
    dotted curves are the ridge-regularized and minimum-norm \(m=40\) EDMD
    fits, which coincide to within \(10^{-4}\); the solid EDMD curve uses
    the first \(\min\{\nop,m\}\) dictionary coordinates. In the legend,
    \(n\) stands for \(\nop\).}
    \label{fig:smalldata}
\end{figure}

\section{Conclusion}\label{sec:conclusion}
We developed a two-stage Galerkin framework, the BSO method, in which a
fixed-lag solution operator of the Kolmogorov backward equation is first
estimated from data and then projected onto an observable space that is
fixed in advance or selected variationally. Since the operator is
approximated before the coordinates are fixed, the same estimator can be
reused across candidate spaces. The invariance of a candidate space is
assessed through conditional means rather than single transitions, and the
Galerkin matrix is defined even when fewer transition pairs than basis
functions are observed.

The error analysis traces the learning error of the estimator, the
sampling error of the Galerkin inner products, and the invariance and
reconstruction properties of the observable space to the computed
eigenvalues, invariant subspaces, and forecasts. The sampling error enters
only through a factor that tends to one. Near an isolated spectral
cluster, the invariance defect is locally equivalent to the distance from
the invariant subspace. The forecast bound is uniform in the horizon when
the measure is invariant and the reduced model preserves constants and is
stable on mean-zero observables. For deterministic dynamics, a learning
rate for FNOs gives one explicit bound for the first stage. 

The numerical
examples apply this framework to spectral and forecasting problems. We compare the Ornstein--Uhlenbeck and triple-well spectra to analytic
and finite-difference references. The Lorenz, SST, and EEG rollouts
are compared with autoregressive neural operators and standard baselines, showing long-rollout stability.
A small-data experiment demonstrates that BSO can result in reduced prediction error than EDMD when fewer transition pairs
than basis functions are observed.

Three limitations indicate directions for further work. First, uniform
control of a general nonlinear operator estimator over the candidate
families is method dependent and is stated here as a hypothesis. Second,
the FNO rate uses exact independent input--output pairs of a deterministic
map and an exact empirical risk minimizer, so that dependent trajectory
data, noisy pairs, the optimization error of training, and stochastic
dynamics are not covered by an explicit rate. Third, partial observations
may identify only a compressed solution operator. A quantitative analysis
of the regime with fewer transition pairs than basis functions is also
left to future work.

\section*{Declarations}
Generative AI tools were used to assist with the code used in the numerical experiments, check computations, and edit the language and organization of the manuscript text. Authors assume responsibility for all content.

\section*{Acknowledgments}
We thank Zongyi Li, Zongren Zou, and Houman Owhadi for helpful discussions.

\appendix
% Make cleveref refer to appendices as "Appendix A", "Appendix D.3"
% independently of the TeX installation (see the remark hook in the preamble).
\crefalias{section}{appendix}
\crefalias{subsection}{subappendix}

\section{Function-space conventions and the learning error}
\label{app:conventions}

Point evaluations are well defined on the space \(\Fn\) of
\cref{sec:data-to-backward}, whereas norms and inner products in
\(L^2(\rho)\) and the operator \(T\) are applied to elements of \(\Fn\)
through their \(\rho\)-equivalence classes. The estimator \(\widehat T\) acts on the
functions themselves and need not extend to \(\cH\); for the composition
estimator of \cref{rem:neural-operator}, \(f\circ\widehat\Phi\) is defined
pointwise on \(\Fn\) but need not be well defined on \(\rho\)-equivalence
classes. For a finite-dimensional subspace \(V\subset\Fn\) on which
\(\|\cdot\|\) is a norm, the restriction \(\widehat T|_V\colon V\to\Fn\) is
regarded as an element of \(\cB(V,\cH)\) through \(\rho\)-equivalence
classes, and the learning error \(\Delta(V)\) of \eqref{eq:learning-error}
is understood in this sense. For a nonlinear estimator
(\cref{rem:nonlinear-estimator}), the first part of the following lemma
applies to the basis that defines \(\widehat T|_V\).

\begin{lemma}\label{lem:basis-linear-surrogate}
Let \(\Psi=(\psi_1,\ldots,\psi_m)\) be a basis of \(V\) with Gram matrix
\(G_\Psi\coloneqq[\ip{\psi_i}{\psi_j}]_{i,j=1}^m\). Then
\begin{equation}\label{eq:surrogate-approximation-error}
    \Delta(V)
    \le
    \lambda_{\min}(G_\Psi)^{-1/2}
    \Bigl(\sum_{j=1}^m
    \|\widehat T\psi_j-T\psi_j\|^2
    \Bigr)^{1/2}.
\end{equation}
Moreover, if \(\ip{\cdot}{\cdot}_n\) is an inner product on
\(V+\widehat TV\), then the matrices \(\widehat K\) and \(\widehat L\) of
\cref{sec:computational-realization} satisfy
\begin{equation}\label{eq:surrogate-defect}
    \sup_{v\in V\setminus\{0\}}
    \frac{\|(I-\widehat P_V)\widehat T v\|_n^2}{\|v\|_n^2}
    =\lambda_{\max}\!\left(\widehat L-\widehat K^\top\widehat K\right).
\end{equation}
\end{lemma}

\begin{proof}
Let \(S\colon\CC^m\to V\), \(Sc\coloneqq\sum_jc_j\psi_j\), so that
\(S^*S=G_\Psi\) and \(SG_\Psi^{-1/2}\) is unitary onto \(V\), and let
\(\mathcal E\coloneqq\widehat T|_V-T|_V\). Then, with
\(\|\cdot\|_{\rm HS}\) the Hilbert--Schmidt norm,
\(\Delta(V)=\|\mathcal E\|\le\|\mathcal ESG_\Psi^{-1/2}\|_{\rm HS}
\le\|\mathcal ES\|_{\rm HS}\,\|G_\Psi^{-1/2}\|_2\) and
\(\|\mathcal ES\|_{\rm HS}^2=\sum_j\|\widehat T\psi_j-T\psi_j\|^2\), which
gives \eqref{eq:surrogate-approximation-error}. For
\eqref{eq:surrogate-defect}, let \(\widehat S\) be the synthesis map of the
computational basis \(\widehat\Psi\), and take adjoints with respect to
\(\ip{\cdot}{\cdot}_n\). Then \(\widehat S\) is unitary from \(\CC^m\) onto
\((V,\ip{\cdot}{\cdot}_n)\), \(\widehat P_V=\widehat S\widehat S^*\),
\(\widehat K=\widehat S^*\widehat T\widehat S\), and
\(\widehat L=(\widehat T\widehat S)^*\widehat T\widehat S\), as in the
residual identity of residual DMD \cite{colbrook2024rigorous}. Hence the
left-hand side of \eqref{eq:surrogate-defect} equals
\(\lambda_{\max}\bigl((\widehat T\widehat S)^*(I-\widehat P_V)\widehat T\widehat S\bigr)
=\lambda_{\max}(\widehat L-\widehat K^\top\widehat K)\).
\end{proof}

\section{Sampling error and data reuse}
\label{app:empirical-concentration}

\begin{lemma}\label{lem:sampling}
Let \(V\subset W\) be finite-dimensional subspaces of \(\Fn\) with
\(V\ne\{0\}\), let \(\|\cdot\|\) be a norm on \(W\), and let
\(\eta\coloneqq\eta(W)<1\). For every \(u\in W\),
\begin{align}
    (1-\eta)\|u\|^2
    &\le\|u\|_n^2
    \le(1+\eta)\|u\|^2,
    \label{eq:sampling-norm}\\
    \|(\widehat P_V-P_V)u\|
    &\le\frac{\eta}{1-\eta}\,\|(I-P_V)u\|,
    \label{eq:sampling-projection}\\
    (1-\eta)\|(I-P_V)u\|^2
    &\le\|(I-\widehat P_V)u\|_n^2
    \le(1+\eta)\|(I-P_V)u\|^2.
    \label{eq:sampling-residual}
\end{align}
\end{lemma}

\begin{proof}
Let \(\chi^W\) be a \(\rho\)-orthonormal basis of \(W\) and
\(u=\sum_ia_i\chi^W_i\). Then \(\|u\|_n^2-\|u\|^2=a^*(\widehat G_W-I)a\),
which gives \eqref{eq:sampling-norm}; in the same way,
\(|\ip{u_1}{u_2}_n-\ip{u_1}{u_2}|\le\eta\|u_1\|\|u_2\|\) for
\(u_1,u_2\in W\), and \(\ip{\cdot}{\cdot}_n\) is an inner product on
\(W\), so that \(\widehat P_V\) is defined there. For
\eqref{eq:sampling-projection}, let \(v\coloneqq P_Vu-\widehat P_Vu\in
V\). Since \(v\) is \(\ip{\cdot}{\cdot}_n\)-orthogonal to
\(u-\widehat P_Vu\) and orthogonal to \(u-P_Vu\), both of which lie in
\(W\),
\[
    \|v\|_n^2
    =-\ip{v}{u-P_Vu}_n
    =-\bigl(\ip{v}{u-P_Vu}_n-\ip{v}{u-P_Vu}\bigr)
    \le\eta\,\|v\|\,\|u-P_Vu\|,
\]
and \eqref{eq:sampling-norm} gives \((1-\eta)\|v\|^2\le\|v\|_n^2\). For
\eqref{eq:sampling-residual}, note that
\(\|(I-\widehat P_V)u\|_n^2=\min_{w\in V}\|u-w\|_n^2\) and
\(\|(I-P_V)u\|^2=\min_{w\in V}\|u-w\|^2\), and apply
\eqref{eq:sampling-norm} to \(u-w\).
\end{proof}

\begin{lemma}\label{lem:objective-equivalence}
Let \(V=V_\theta\) with \(\eta\coloneqq\eta_V<1\). Then
\begin{equation}\label{eq:objective-equivalence}
    \cJ(V)^{1/2}
    \le
    \kappa(\eta)\,\widehat\cJ(\theta)^{1/2}
    +\Delta(V),
    \qquad
    \widehat\cJ(\theta)^{1/2}
    \le
    \kappa(\eta)\bigl(\cJ(V)^{1/2}+\Delta(V)\bigr).
\end{equation}
\end{lemma}

\begin{proof}
Let
\(\cI_{\widehat T}(V)\coloneqq\|(I-P_V)\widehat T|_V\|_{\cB(V,\cH)}\) and
\(\cJ_{\widehat T}(V)\coloneqq\cI_{\widehat T}(V)^2+\alpha\cR(V)^2\). Since
\(\eta<1\), \eqref{eq:sampling-norm} shows that \(\ip{\cdot}{\cdot}_n\) is
an inner product on \(W_V\supset V+\widehat TV\). By
\eqref{eq:surrogate-defect}, the first term of \(\widehat\cJ(\theta)\)
therefore equals
\(\sup_{v\ne0}\|(I-\widehat P_V)\widehat T v\|_n^2/\|v\|_n^2\). By
\eqref{eq:sampling-residual} and \eqref{eq:sampling-norm}, the numerator lies between
\((1\mp\eta)\|(I-P_V)\widehat T v\|^2\) and the denominator between
\((1\mp\eta)\|v\|^2\), so that this term lies between
\(\kappa(\eta)^{-2}\cI_{\widehat T}(V)^2\) and
\(\kappa(\eta)^2\cI_{\widehat T}(V)^2\). By
\eqref{eq:sampling-residual}, also
\((1-\eta)\cR(V)^2\le\max_j\|(I-\widehat P_V)g_j\|_n^2\le(1+\eta)\cR(V)^2\).
Hence
\(\kappa(\eta)^{-2}\cJ_{\widehat T}(V)\le\widehat\cJ(\theta)
\le\kappa(\eta)^2\cJ_{\widehat T}(V)\). Finally,
\(|\cI_{\widehat T}(V)-\cI(V)|\le\|(I-P_V)(\widehat T|_V-T|_V)\|
\le\Delta(V)\), and since \(t\mapsto(t^2+\alpha\cR(V)^2)^{1/2}\) is
\(1\)-Lipschitz, \(|\cJ_{\widehat T}(V)^{1/2}-\cJ(V)^{1/2}|\le\Delta(V)\).
Combining these estimates gives \eqref{eq:objective-equivalence}.
\end{proof}

\paragraph{Concentration}
For a fixed space \(W\), independent sample points with law \(\rho\), and
the weights \(w_k=1/n\), the matrix \(\widehat G_W-I\) is a sum of \(n\)
independent centered matrices. The matrix Bernstein inequality
\cite[Theorem~1.4]{tropp2012user} then gives, for \(\nu\in(0,1)\), with
probability at least \(1-\nu\),
\[
    \eta(W)
    \le
    \Bigl(\frac{2c_W\log(2\dim W/\nu)}{n}\Bigr)^{1/2}
    +\frac{2c_W\log(2\dim W/\nu)}{3n},
\]
where \(c_W\) bounds \(\|\chi^W(x)\|_2^2\) for \(\rho\)-almost every \(x\);
bounds of this type for the empirical Gram matrix of a least-squares space
are given in \cite{cohen2013stability,cohen2017optimal}. For stationary
ergodic trajectories and bounded \(\chi^W\), Bernstein inequalities for
Markov chains with a spectral gap, or a pseudo-spectral gap in the
nonreversible case \cite{paulin2015concentration}, give
\(\eta(W)=O(n^{-1/2})\) with high probability, with constants depending on
\(W\) and on the gap.

\paragraph{Uniform control}
The quantities \(\eta_\Theta\) and \(\Delta_\Theta\) of \cref{thm:main}
reduce to quantities for a fixed space when the estimator is linear and all
candidate spaces lie in a fixed finite-dimensional space \(U\subset\Fn\),
such as the polynomials of a given degree, on which \(\|\cdot\|\) is a norm.
Then \(W_{V_\theta}\subset W_U\coloneqq U+\widehat TU+\spanop\cG\) for every
\(\theta\in\Theta\), and since \(\eta\) and \(\Delta\) are monotone with
respect to inclusion, \(\eta_\Theta\le\eta(W_U)\) and
\(\Delta_\Theta\le\Delta(U)\). The same statements apply to weighted nodes
(\cref{sec:computational-realization}), for which \(\eta\) is the quadrature
error of the Gram matrix of \(W_V\).

\section{Proofs of the main results}\label{app:proofs}

\subsection{Proof of \texorpdfstring{\cref{prop:one-step}}{Proposition 4.1}}
\label{app:proof-one-step}
\begin{proof}
For \(v\in V\),
\[
    \widehat\cK v-P_VTv
    =P_V(\widehat T-T)v+(\widehat P_V-P_V)\widehat T v .
\]
The first term has norm at most \(\Delta(V)\|v\|\). Since
\(\widehat T v\) belongs to \(W_V\)
and
\(\|(I-P_V)\widehat T v\|\le\|(I-P_V)Tv\|+\|(\widehat T-T)v\|
\le(\cI(V)+\Delta(V))\|v\|\), \eqref{eq:sampling-projection} bounds the second
term by \(\eta(1-\eta)^{-1}(\cI(V)+\Delta(V))\|v\|\). Adding the two
bounds gives \eqref{eq:compression-error}, and
\((T-\widehat\cK)v=(I-P_V)Tv+(P_VT-\widehat\cK)v\) gives
\eqref{eq:one-step-error}, because
\(\cI(V)+(\Delta(V)+\eta\cI(V))/(1-\eta)=(\cI(V)+\Delta(V))/(1-\eta)\).
\end{proof}

\subsection{Proof of \texorpdfstring{\cref{thm:main}}{Theorem 4.2}}
\label{app:proof-main}
\begin{proof}
Let \(\theta\in\Theta\) and \(\kappa\coloneqq\kappa(\eta_\Theta)\ge1\). Since
\(\eta_{\widehat V}\le\eta_\Theta\) and
\(\Delta(\widehat V)\le\Delta_\Theta\), the first inequality in
\eqref{eq:objective-equivalence} of \cref{lem:objective-equivalence} for
\(\theta=\widehat\theta\) gives
\(\cJ(\widehat V)^{1/2}\le\kappa\,\widehat\cJ(\widehat\theta)^{1/2}
+\Delta_\Theta\). By
\eqref{eq:near-minimizer} and the second inequality in
\eqref{eq:objective-equivalence},
\(\widehat\cJ(\widehat\theta)^{1/2}\le\widehat\cJ(\theta)^{1/2}+\zeta^{1/2}
\le\kappa(\cJ(V_\theta)^{1/2}+\Delta_\Theta)+\zeta^{1/2}\). Combining the
two bounds gives
\(\cJ(\widehat V)^{1/2}\le\kappa^2(\cJ(V_\theta)^{1/2}+\Delta_\Theta)
+\kappa\zeta^{1/2}+\Delta_\Theta\), which implies \eqref{eq:main-bound}.
\end{proof}

\subsection{Proof of \texorpdfstring{\cref{thm:galerkin-spectral}}{Theorem 4.3}}
\label{app:proof-galerkin-spectral}
\begin{proof}
Let \(R\coloneqq P_V\iota\colon E_\star\to V\). For \(u\in E_\star\),
\(\|Ru-u\|=\|(I-P_V)u\|\le d\|u\|\), and
\(R^*R=I-\iota^*(I-P_V)\iota\) satisfies \((1-d^2)I\le R^*R\le I\) because
\(\|(I-P_V)\iota\|=d\). Hence \(R\) is invertible, its singular values lie
in \([(1-d^2)^{1/2},1]\), and \(\cQ\coloneqq R(R^*R)^{-1/2}\) is an
isometric isomorphism of \(E_\star\) onto \(V\) with
\[
    \|\cQ-\iota\|
    \le\|R((R^*R)^{-1/2}-I)\|+\|R-\iota\|
    \le1-(1-d^2)^{1/2}+d
    \le d^2+d
    \le2d ,
\]
since \(\|R((R^*R)^{-1/2}-I)\|=\|I-(R^*R)^{1/2}\|\). Regarding \(\cQ\) as
a map into \(\cH\), we have \(\cQ^*P_V=\cQ^*\), and
\(T|_{E_\star}=\iota^*T\iota\) by invariance. Therefore
\[
    \cQ^*\widehat\cK\cQ-T|_{E_\star}
    =\cQ^*\bigl(\widehat\cK-P_VT|_V\bigr)\cQ
    +\bigl(\cQ^*T\cQ-\iota^*T\iota\bigr).
\]
By \eqref{eq:compression-error} and the upper bound \(\cI(V)\le2M_Td\) of
\eqref{eq:local-identification}, which holds for every \(V\), the first
term is bounded by
\(\|\widehat\cK-P_VT|_V\|\le(\Delta(V)+2\eta M_Td)/(1-\eta)\), and the
second by \(2M_T\|\cQ-\iota\|\le4M_Td\). Adding the two bounds and using
\(2\eta+4(1-\eta)\le4\) gives \eqref{eq:aligned-matrix-bound}.

For the convergence statement, recall that the eigenvalues of \(\widehat K\)
are those of \(\widehat\cK\), and hence those of
\(\cQ^*\widehat\cK\cQ=\cQ^{-1}\widehat\cK\cQ\). For large \(k\),
\(\dist(V_k,E_\star)<1\), and the first part gives isometries \(\cQ_k\) with
\(\|\cQ_k-\iota\|\to0\) and, since \(\sup_k\eta_{V_k}<1\),
\(A_k\coloneqq\cQ_k^*\widehat\cK_k\cQ_k\to T|_{E_\star}\) in
\(\cB(E_\star)\). Since \(E_\star\) is finite-dimensional, the eigenvalues of
\(A_k\) converge to those of \(T|_{E_\star}\) with algebraic multiplicity,
and the Riesz projections of \(A_k\) for the eigenvalues that converge to a
given part of \(\spec(T|_{E_\star})\) converge in norm to the Riesz
projection of \(T|_{E_\star}\) for that part
\cite[Chap.~II, \S5.1]{kato1966perturbation}, \cite{chatelin2011spectral};
hence their ranges \(F_k\) converge in \(\dist\) to its range \(F\). The
corresponding spectral subspaces of \(\widehat\cK_k\) are \(\cQ_kF_k\), and
\(\|P_{\cQ_kF_k}-P_{F_k}\|\le2\|\cQ_k-\iota\|\to0\) because \(\cQ_k\) is an
isometry, so that \(\dist(\cQ_kF_k,F)\to0\). In the rank-one case,
normalized eigenvectors converge up to unimodular factors.
\end{proof}

\subsection{Proof of \texorpdfstring{\cref{thm:quantitative-identification}}{Theorem 4.4}}
\label{app:proof-local-identification}
\begin{proof}
Let \(d\coloneqq\dist(V,E_\star)\). For the upper bound, write
\(TP_V=TP_\star P_V+T(I-P_\star)P_V\) and use \(TP_\star=P_\star TP_\star\),
which gives
\[
    \cI(V)
    =\|(I-P_V)TP_V\|
    \le\|(I-P_V)P_\star\|\,M_T+M_T\,\|(I-P_\star)P_V\|
    =2M_Td .
\]
For the lower bound, \(d\le r_\star/\sqrt2<1\), so that \(P_\star\) maps
\(V\) bijectively onto \(E_\star\). Hence \(V=\{u+Zu:u\in E_\star\}\) for
some \(Z\in\cB(E_\star,E_\star^\perp)\), and
\(d=\|(I-P_\star)P_V\|=\|Z\|(1+\|Z\|^2)^{-1/2}\), because
\(\|(I-P_\star)v\|^2/\|v\|^2=\|Zu\|^2/(\|u\|^2+\|Zu\|^2)\) for
\(v=u+Zu\), so that
\(t\coloneqq\|Z\|=d(1-d^2)^{-1/2}\), the tangent of the largest principal
angle, satisfies \(d\le t\le\sqrt2\,d\le r_\star\le1\). For \(v=u+Zu\), the
invariance of \(E_\star\) gives
\[
    Tv=u'+Zu'+\Ric(Z)u,
    \qquad
    u'\coloneqq(A_\star+B_\star Z)u\in E_\star,
    \qquad
    \Ric(Z)\coloneqq D_\star Z-ZA_\star-ZB_\star Z,
\]
so that \((I-P_V)Tv=(I-P_V)\Ric(Z)u\) with \(\Ric(Z)u\in E_\star^\perp\).
Moreover, \(\|v\|\le(1+t^2)^{1/2}\|u\|\), and for \(w\in E_\star^\perp\)
and \(\tilde u\in E_\star\),
\(\|w\|\le\|w-Z\tilde u\|+t\|\tilde u\|
\le(1+t^2)^{1/2}\|w-\tilde u-Z\tilde u\|\), so that
\(\|(I-P_V)w\|\ge(1+t^2)^{-1/2}\|w\|\). Therefore
\(\cI(V)\ge\|\Ric(Z)\|/(1+t^2)\). Since \(\|D_\star Z-ZA_\star\|\ge\delta_\star t\)
and \(\|ZB_\star Z\|\le\|B_\star\|t^2\le\delta_\star t/2\) for
\(t\le r_\star\), we obtain \(\|\Ric(Z)\|\ge\delta_\star t/2\), and
\(1+t^2\le2\) gives \(\cI(V)\ge\delta_\star t/4\ge\delta_\star d/4\).
\end{proof}

\subsection{Proof of \texorpdfstring{\cref{thm:forecast}}{Theorem 4.5}}\label{app:proof-forecast}
\begin{proof}
Write \(P\coloneqq P_V\) and \(\widehat P\coloneqq\widehat P_V\). Then
\[
    T^Nf-f^{(N)}
    =T^N(I-P)f+\sum_{k=0}^{N-1}T^{N-1-k}(T-\widehat\cK)\widehat\cK^kPf
    +\widehat\cK^N(P-\widehat P)f .
\]
The first term is bounded by \(M_T^N\|(I-P)f\|\). Since
\(\widehat\cK^kPf\in V\), \eqref{eq:one-step-error} bounds the \(k\)th
summand by \(M_T^{N-1-k}e(V)\,\|\widehat\cK\|^k\|Pf\|\). Since
\(f\in W_V\), \eqref{eq:sampling-projection} bounds the last term by
\(\|\widehat\cK\|^N\eta(1-\eta)^{-1}\|(I-P)f\|\). Since
\(M_T^{N-1-k}\|\widehat\cK\|^k\le M^{N-1}\) and
\(M_T^N+\|\widehat\cK\|^N\eta/(1-\eta)\le M^N/(1-\eta)\), adding the three
bounds gives \eqref{eq:matrix-power-forecast-bound}.

For the selected space \(\widehat V\) and \(f\in\cG\), \cref{thm:main} gives
\(\|(I-P_{\widehat V})f\|\le\cR(\widehat V)
\le(\cJ(\widehat V)/\alpha)^{1/2}\le\varepsilon_\Theta/\alpha^{1/2}\) and
\(\cI(\widehat V)\le\cJ(\widehat V)^{1/2}\le\varepsilon_\Theta\), so that
\(e(\widehat V)\le(\varepsilon_\Theta+\Delta_\Theta)/(1-\eta_\Theta)\).
Moreover, \(\eta_{\widehat V}\le\eta_\Theta\) and
\(\|P_{\widehat V}f\|\le\|f\|\).
Inserting these bounds in \eqref{eq:matrix-power-forecast-bound} gives
\eqref{eq:selected-forecast-bound}.
\end{proof}

\section{Supplements to the error analysis}
\label{app:supplements}

\subsection{Identification of the target subspace}
\label{app:identification}

\paragraph{Separation}
If \(E_\star\) is the spectral subspace of an isolated cluster
\(\Lambda_\star\) (\cref{sec:spec_pert_theory}), then \(\delta_\star>0\).
Indeed, \(\spec(A_\star)=\Lambda_\star\), and
\(\spec(D_\star)=\spec(T)\setminus\Lambda_\star\) because the restriction
of \(T\) to \(\Ran(I-\Pi)\) has this spectrum
\cite[Chap.~III, \S6.4]{kato1966perturbation} and is similar to \(D_\star\)
through \(I-P_\star\); the map \(Z\mapsto D_\star Z-ZA_\star\) is therefore
invertible by Rosenblum's theorem
\cite{rosenblum1956operator,bhatia1997how}. If \(T\) is self-adjoint and
\(\Lambda_\star\) consists of the \(m\) largest eigenvalues, then
\(B_\star=0\), \(r_\star=1\), and \(\delta_\star\) is the distance between
\(\Lambda_\star\) and the rest of the spectrum \cite{bhatia1997how}; in
general, \(\delta_\star\) also depends on the departure of \(T\) from
normality.

\paragraph{Resolution}
Let
\[
    s_\star
    \coloneqq
    \min\Bigl\{\Bigl(\sum_{j=1}^q\ip{u}{g_j}^2\Bigr)^{1/2}:
    u\in E_\star,\ \|u\|=1\Bigr\};
\]
it measures how well the reconstruction observables resolve \(E_\star\), and
the condition \(s_\star>0\) means that their projections onto \(E_\star\)
span \(E_\star\).

\begin{lemma}[Localization by the reconstruction term]\label{lem:localization}
For every \(m\)-dimensional subspace \(V\subset\cH\),
\begin{equation}\label{eq:localization}
    s_\star\,\dist(V,E_\star)
    \le
    \sqrt q\,\bigl(\cR(V)+\cR(E_\star)\bigr).
\end{equation}
In particular, if \(\cG\subset E_\star\) and \(s_\star>0\), then
\(\dist(V,E_\star)\le(\sqrt q/s_\star)(\cJ(V)/\alpha)^{1/2}\).
\end{lemma}

\begin{proof}
Let \(d\coloneqq\dist(V,E_\star)\); for \(d=0\) there is nothing to prove.
Since \(E_\star\) and \(V\) have the same finite dimension,
\(d=\|(I-P_V)P_\star\|\), and there is a unit vector \(u\in E_\star\) with
\(\|(I-P_V)u\|=d\). Set \(w\coloneqq(I-P_V)u/d\). Then \(w\perp V\),
\(\|w\|=1\), and \(\ip{w}{u}=\|(I-P_V)u\|^2/d=d\), so that
\(d\le\|P_\star w\|\). By the definition of \(s_\star\),
\(s_\star\|P_\star w\|\le(\sum_j\ip{P_\star w}{g_j}^2)^{1/2}\), and, since
\(w\perp V\),
\[
    \ip{P_\star w}{g_j}
    =\ip{w}{(I-P_V)g_j}-\ip{w}{(I-P_\star)g_j},
    \qquad
    j=1,\ldots,q .
\]
Hence \(|\ip{P_\star w}{g_j}|\le\cR(V)+\cR(E_\star)\), so that
\(s_\star d\le s_\star\|P_\star w\|\le\sqrt q\,(\cR(V)+\cR(E_\star))\).
If \(\cG\subset E_\star\), then \(\cR(E_\star)=0\), and
\(\cR(V)\le(\cJ(V)/\alpha)^{1/2}\) gives the last claim.
\end{proof}

If \(\cG\subset E_\star\) and \(s_\star>0\), the objective thus identifies
\(E_\star\) globally, without any assumption on the spectrum.

\begin{corollary}[Identification from the objective]\label{cor:identification}
Let \(\delta_\star>0\) and \(s_\star>0\), and let \(V\) be an
\(m\)-dimensional subspace with
\begin{equation}\label{eq:basin-condition}
    \frac{\sqrt q}{s_\star}
    \Bigl(\Bigl(\frac{\cJ(V)}{\alpha}\Bigr)^{1/2}+\cR(E_\star)\Bigr)
    \le
    \frac{r_\star}{\sqrt2}.
\end{equation}
Then
\begin{equation}\label{eq:quant-id-a}
    \dist(V,E_\star)
    \le
    \frac{4\,\cI(V)}{\delta_\star}
    \le
    \frac{4\,\cJ(V)^{1/2}}{\delta_\star}.
\end{equation}
In particular, if \(\eta_\Theta<1\) and \eqref{eq:basin-condition} holds
with \(\cJ(V)^{1/2}\) replaced by \(\varepsilon_\Theta\), then the selected
space \(\widehat V\) of \cref{thm:main} satisfies
\(\dist(\widehat V,E_\star)\le4\varepsilon_\Theta/\delta_\star\), and the
isometry \(\cQ\) of \cref{thm:galerkin-spectral} for \(V=\widehat V\)
satisfies
\begin{equation}\label{eq:selected-spectral-bound}
    \bigl\|\cQ^*\widehat\cK\cQ-T|_{E_\star}\bigr\|
    \le
    \frac{\Delta_\Theta+16M_T\,\varepsilon_\Theta/\delta_\star}{1-\eta_\Theta}.
\end{equation}
\end{corollary}

\begin{proof}
By \eqref{eq:localization} and \(\cR(V)\le(\cJ(V)/\alpha)^{1/2}\), condition
\eqref{eq:basin-condition} gives \(\dist(V,E_\star)\le r_\star/\sqrt2\), and
the lower bound in \eqref{eq:local-identification} together with
\(\cI(V)\le\cJ(V)^{1/2}\) gives \eqref{eq:quant-id-a}. For the selected
space, \(\cJ(\widehat V)^{1/2}\le\varepsilon_\Theta\) by \cref{thm:main}, so
that \eqref{eq:basin-condition} holds for \(V=\widehat V\) and
\(\dist(\widehat V,E_\star)\le4\varepsilon_\Theta/\delta_\star\). Since
\(\dist(\widehat V,E_\star)\le r_\star/\sqrt2<1\),
\(\eta_{\widehat V}\le\eta_\Theta\), and
\(\Delta(\widehat V)\le\Delta_\Theta\), \eqref{eq:aligned-matrix-bound} for
\(V=\widehat V\) gives \eqref{eq:selected-spectral-bound}.
\end{proof}

Since \(\varepsilon_\Theta\ge2\Delta_\Theta+\zeta^{1/2}\), the hypothesis on
the selected space requires \(\alpha^{-1/2}(\Delta_\Theta+\zeta^{1/2})\) to
be small, so that the weight \(\alpha\) cannot be chosen small relative to
the squared learning error and the tolerance. Condition
\eqref{eq:basin-condition} also requires the reconstruction observables to
lie close to \(E_\star\), so that identification uses prior information
about the target through \(\cG\).

\paragraph{Bias}
If \(\cG\not\subset E_\star\), the reconstruction term may bias the
minimizers of \(\cJ\) away from \(E_\star\), and without a condition on
\(\cR(E_\star)\) they may converge to a different invariant subspace as
\(\alpha\to0\). For an irreducible two-state Markov chain with invariant
\(\rho\), a nonconstant eigenfunction \(u\) of \(T\) with \(\|u\|=1\),
\(E_\star=\spanop\{u\}\), and \(\cG=\{2\cdot1+u\}\), one has
\(\cJ(E_\star)=4\alpha>\alpha=\cJ(\spanop\{1\})\), and the minimizers
converge to \(\spanop\{1\}\). Condition \eqref{eq:basin-condition} excludes
this case, since it requires \(\cR(E_\star)\) to be small relative to
\(s_\star r_\star\).

\subsection{Residual bound for computed eigenpairs}
\label{app:residual}

Individual computed eigenpairs can be checked a posteriori by a residual of
the form used in residual DMD \cite{colbrook2024rigorous}, computed from the
matrices \(\widehat K\) and \(\widehat L\) of
\cref{sec:computational-realization}.

\begin{proposition}\label{prop:eigenpair-certificate}
Let \(\eta\coloneqq\eta_V<1\). For \(\mu\in\CC\) and
\(c\in\CC^m\setminus\{0\}\), let \(\varphi\coloneqq\sum_jc_j\widehat\psi_j\) and
\[
    \widehat{\res}(\mu,c)
    \coloneqq
    \|c\|_2^{-1}
    \bigl(c^*\widehat Lc
    -2\operatorname{Re}(\overline\mu\,c^*\widehat Kc)
    +|\mu|^2\|c\|_2^2\bigr)^{1/2}.
\]
Then
\begin{equation}\label{eq:eigenpair-certificate}
    \frac{\|T\varphi-\mu\varphi\|}{\|\varphi\|}
    \le
    \Delta(V)
    +\kappa(\eta)\,\widehat{\res}(\mu,c).
\end{equation}
\end{proposition}

\begin{proof}
Since \(\widehat\Psi\) is orthonormal in \(\ip{\cdot}{\cdot}_n\),
\(\|\varphi\|_n=\|c\|_2\) and
\(\|\widehat T\varphi-\mu\varphi\|_n^2
=c^*\widehat Lc-2\operatorname{Re}(\overline\mu\,c^*\widehat Kc)+|\mu|^2\|c\|_2^2\),
which is the residual identity of residual DMD \cite{colbrook2024rigorous}
with the values of the estimator in place of observed values; thus
\(\widehat{\res}(\mu,c)=\|\widehat T\varphi-\mu\varphi\|_n/\|\varphi\|_n\).
Since \(\widehat T\varphi-\mu\varphi\) and \(\varphi\) belong to
\(W_V\), \eqref{eq:sampling-norm} gives
\(\|\widehat T\varphi-\mu\varphi\|\le\kappa(\eta)\,
\widehat{\res}(\mu,c)\,\|\varphi\|\), and
\(\|T\varphi-\mu\varphi\|\le\|(T-\widehat T)\varphi\|
+\|\widehat T\varphi-\mu\varphi\|\) gives
\eqref{eq:eigenpair-certificate}.
\end{proof}

The quantity
\(\widehat{\res}(\mu,c)=\|\widehat T\varphi-\mu\varphi\|_n/\|\varphi\|_n\)
is computable; for an eigenpair \(\widehat Kc=\widehat\mu c\),
\(\widehat{\res}(\widehat\mu,c)^2=c^*\widehat Lc/\|c\|_2^2-|\widehat\mu|^2\).
For normal \(T\), \eqref{eq:eigenpair-certificate} bounds the distance from
\(\mu\) to \(\spec(T)\); in general, it gives the corresponding
pseudospectral inclusion \cite{trefethen2005spectra}. As in residual DMD
\cite{colbrook2024rigorous}, one may retain only the computed eigenpairs
whose right-hand side is below a prescribed tolerance; in contrast with the
single-transition residuals of \cref{rem:stochastic-residual-bias},
\(\widehat{\res}\) contains no conditional-variance term. Negative computed
multipliers of a reversible diffusion can be discarded directly, since
\(\spec(T_\tau)\subset[0,1]\) if \(\cL\) is self-adjoint and nonpositive.

\subsection{Forecast bounds}
\label{app:forecast-supplements}

The norm of the computed Galerkin operator satisfies
\(\|\widehat\cK\|\le\kappa(\eta_V)\,\|\widehat K\|_2\) by \eqref{eq:sampling-norm},
because \(\|v\|_n=\|c\|_2\) for \(v=\sum_jc_j\widehat\psi_j\). The bound
\eqref{eq:matrix-power-forecast-bound} also applies with the compressed
operator \(T_\tau^{\rm obs}\) of \cref{app:partial-observation} in place of
\(T\); its powers are the multistep targets under the Markov condition
stated there.

Suppose that \(\rho\) is invariant, so that \(T\) is a contraction with
\(T1=1\), and that \(1\in V\). Let
\(V_0\coloneqq\{v\in V:\ip{1}{v}=0\}\), and let
\(f_0\coloneqq f-\ip{1}{f}1\) denote the mean-zero part of \(f\).

\begin{corollary}[Forecast bound uniform in the horizon]\label{cor:uniform-forecast}
In the setting of \cref{thm:forecast}, with \(\rho\) invariant and
\(1\in V\), assume in addition that
\(\widehat\cK1=1\), \(\widehat\cK V_0\subset V_0\), and
\(\|\widehat\cK^k|_{V_0}\|\le C_Kr_K^k\) for \(k\ge0\), with \(C_K\ge1\) and
\(0\le r_K<1\).
Then, for every \(N\ge1\),
\begin{equation}\label{eq:uniform-forecast-bound}
    \|T^Nf-f^{(N)}\|
    \le
    C_K\Bigl(\frac{\|(I-P_V)f_0\|}{1-\eta_V}
    +\frac{e(V)\,\|P_Vf_0\|}{1-r_K}\Bigr).
\end{equation}
\end{corollary}

\begin{proof}
Write \(P\coloneqq P_V\) and \(\widehat P\coloneqq\widehat P_V\). Since
\(T1=1\), \(\widehat\cK1=1\), and \(P1=\widehat P1=1\),
\[
    T^Nf-f^{(N)}
    =T^N(I-P)f_0
    +\sum_{k=0}^{N-1}T^{N-1-k}(T-\widehat\cK)\widehat\cK^kPf_0
    -\widehat\cK^N(\widehat P-P)f_0 .
\]
Because \(1\in V\), \(Pf_0\in V_0\). Hence
\(\widehat\cK^kPf_0\in V_0\) has norm at most \(C_Kr_K^k\|Pf_0\|\), and
\((T-\widehat\cK)\widehat\cK^kPf_0\) has norm at most
\(e(V)C_Kr_K^k\|Pf_0\|\) by \eqref{eq:one-step-error}. Since \(V=\spanop\{1\}\oplus V_0\)
orthogonally and \(\widehat\cK\) leaves both summands invariant,
\(\|\widehat\cK^N\|\le C_K\), and \eqref{eq:sampling-projection} bounds the
last term by \(C_K\eta_V(1-\eta_V)^{-1}\|(I-P)f_0\|\). Since \(T\) is a
contraction, the first term is at most \(\|(I-P)f_0\|\), and the sum is at
most \(C_Ke(V)(1-r_K)^{-1}\|Pf_0\|\). Finally,
\(1+C_K\eta_V/(1-\eta_V)\le C_K/(1-\eta_V)\).
\end{proof}

The decay condition holds, with a suitable \(C_K\), for every \(r_K\)
larger than the spectral radius of \(\widehat\cK\) on \(V_0\). The condition
\(\widehat\cK1=1\) holds when \(1\in V\) and \(\widehat T1=1\), as for
composition estimators. The condition \(\widehat\cK V_0\subset V_0\) holds
for the exact compression \(P_VT|_V\), but in general not for
\(\widehat\cK\), and it involves the \(\rho\)-means of the basis functions.
From the sample, one can instead impose the invariance of
\(V_{0,n}\coloneqq\{v\in V:\ip{1}{v}_n=0\}\). If the computational basis
is chosen with \(\widehat\psi_1=1\), this condition and \(\widehat\cK1=1\)
state that the first column and the first row of \(\widehat K\) are those of
the identity matrix. With \(V_{0,n}\) in place of \(V_0\) in the hypotheses
of \cref{cor:uniform-forecast}, \eqref{eq:uniform-forecast-bound} holds with
\(C_K\) replaced by \(2\kappa(\eta_V)C_K\), because every \(v\in V\) is
the sum of a constant and an element of \(V_{0,n}\), both of norm at most
\(\kappa(\eta_V)\|v\|\) by \eqref{eq:sampling-norm}, and
\((T-\widehat\cK)1=0\), so that the constant component of
\(\widehat\cK^kP_Vf_0\) does not contribute to the sum in the proof. The
decay hypothesis is stated in the \(L^2(\rho)\) norm; a bound read off
from \(\widehat K\) in the Euclidean norm enters with a further factor
\(\kappa(\eta_V)\) by \eqref{eq:sampling-norm}.

\begin{remark}\label{rem:postprocessing}
A modification \(\widehat\cK\mapsto\widehat\cK'\) of the computed Galerkin
operator, such as a rescaling of spectral blocks or the constraints above,
enters \eqref{eq:one-step-error}, \eqref{eq:aligned-matrix-bound}, and the
forecast bounds through the additional term \(\|\widehat\cK'-\widehat\cK\|\),
with \(\|\widehat\cK\|\) replaced by \(\|\widehat\cK'\|\) in \(M\). Replacing \(\widehat K\) by
its symmetric part changes \(\widehat\cK\) by at most
\(\kappa(\eta_V)\|\widehat K-\widehat K^\top\|_2/2\), which is small together
with \(\Delta(V)\) and \(\eta_V\) when \(T\) is self-adjoint on
\(L^2(\rho)\). If the component \((I-\widehat P_V)f\) is propagated by an
operator \(U_N\) instead of being discarded, the forecast error is at most
\(NM^{N-1}e(V)\,\|\widehat P_Vf\|+\|{(T^N-U_N)(I-\widehat P_V)f}\|\).
\end{remark}

For measure-preserving deterministic dynamics, \(T\) is an isometry, whereas
the forecast of a stable reduced model relaxes to a constant, which is zero
if the spectral radius of \(\widehat\cK\) is less than one, as for the
Lorenz systems (\cref{app:numerical-details}). The forecast error then
tends, as \(N\to\infty\), to a limit that is at least the standard deviation
\(\|f_0\|\) of \(f\), the error of a constant forecast; the plateaus at the
long leads in \cref{fig:lorenz-forecast} are consistent with this.

\section{Partial observations}\label{app:partial-observation}

Let \(y_t\coloneqq h(X_t)\) be the observed variables
(\cref{rem:neural-operator}), and let \(\cH_Y\subset\cH\) be the closed
subspace of observables of the form \(f\circ h\). If \(X_t\sim\rho\), then,
for \(\varphi\in\cH_Y\), the Markov and tower properties give
\[
    \EE\!\left[\varphi(X_{t+\tau})\mid y_t\right]
    =\bigl(P_{\cH_Y}T_\tau\varphi\bigr)(X_t).
\]
Thus partially observed lagged data identify the compressed operator
\[
    T_\tau^{\rm obs}
    \coloneqq
    P_{\cH_Y}T_\tau|_{\cH_Y}.
\]
If a \(T_\tau\)-invariant subspace \(E\) lies in \(\cH_Y\), then
\(T_\tau^{\rm obs}|_E=T_\tau|_E\); otherwise, the compressed and full-state
spectra may differ. For \(\varphi\in\cH_Y\) and \(\mu\in\CC\),
\[
    \|T_\tau\varphi-\mu\varphi\|^2
    =\|T_\tau^{\rm obs}\varphi-\mu\varphi\|^2
     +\|(I-P_{\cH_Y})T_\tau\varphi\|^2.
\]
The powers of \(T_\tau^{\rm obs}\) represent the multistep observed
evolution only when the sampled observation process is Markov at lag
\(\tau\). Delay coordinates may recover the deterministic state under an
embedding hypothesis, or more generally define an enlarged observable
process
\cite{takens1981detecting,arbabi2017ergodic,kamb2020time,giannakis2019data,das2019delay}.

A map \(\widehat\Phi\colon Y\to Y\) trained by least squares on the lagged
observations \((y_t,y_{t+\tau})\) approximates the conditional mean
\(y\mapsto\EE[y_{t+\tau}\mid y_t=y]\), so that \(\widehat\Phi\circ h\)
approximates \(T_\tau^{\rm obs}\) applied to the components of \(h\). If the
observed variables do not evolve autonomously and \(f\) is an observable on
\(Y\) that is not affine, then
\(f(\EE[y_{t+\tau}\mid y_t])\ne\EE[f(y_{t+\tau})\mid y_t]\) in general, so
that \(f\circ\widehat\Phi\circ h\) does not approximate
\(T_\tau^{\rm obs}(f\circ h)\); an estimate of \(T_\tau^{\rm obs}(f\circ h)\)
is then obtained from the conditional law of \(y_{t+\tau}\) given \(y_t\)
(\cref{app:fno-setting}) or by regressing \(f(y_{t+\tau})\) on \(y_t\).

\section{Learning error for composition estimators and the FNO setting}
\label{app:fno-setting}

\paragraph{Composition estimators}
Let \(\widehat Tf=f\circ\widehat\Phi\) be the composition estimator of
\cref{rem:neural-operator}, let \(\dist_\Omega\) be a metric on
\(\Omega=Y\), and let \(\Omega_0\subset\Omega\) contain \(\Phi(y)\) and
\(\widehat\Phi(y)\) for \(\rho\)-almost every \(y\). We write
\(\Lip(\,\cdot\,)\) for the Lipschitz constant on \(\Omega_0\) of a map into
a Euclidean space, with respect to \(\dist_\Omega\) and the Euclidean norm.
For an \(m\)-dimensional subspace \(V\subset\Fn\) with \(\rho\)-orthonormal
basis \(\chi\), set
\[
    \Lip(V)\coloneqq\Lip(\chi),
    \qquad
    \chi\colon y\mapsto(\chi_1(y),\ldots,\chi_m(y))^\top\in\mathbb R^m .
\]
This quantity does not depend on the choice of \(\chi\), and it is the
largest Lipschitz constant on \(\Omega_0\) of a function \(v\in V\) with
\(\|v\|=1\). Since
\(|v(\widehat\Phi(y))-v(\Phi(y))|\le\Lip(V)\,\|v\|\,
\dist_\Omega(\widehat\Phi(y),\Phi(y))\) for \(v\in V\), taking the
\(L^2(\rho)\) norm gives \eqref{eq:composition-learning-error}. For
stochastic dynamics, \(y_{t+\tau}\) is not a function of \(y_t\). If
\(f\circ\widehat\Phi\) is then replaced by the mean of \(f\) under a learned
transition kernel, such as a generative model of the next observation, and
\(\Omega_0\) contains the supports of the learned and true transition kernels
for \(\rho\)-almost every \(y\), the
Kantorovich--Rubinstein duality \cite{villani2009optimal} gives
\eqref{eq:composition-learning-error} with
\(\dist_\Omega(\widehat\Phi,\Phi)\) replaced by the Wasserstein-1 distance
between the learned and true transition kernels.

\paragraph{FNO setting}
For \cref{cor:fno-rate}, take
\[
    \Omega=L^2(\mathbb T^{d_x};\mathbb R^{d_{\rm out}}),
    \qquad
    \Ba
    =\{u\in H^a(\mathbb T^{d_x};\mathbb R^{d_{\rm out}}):
      \|u\|_{H^a}\le c_a\},
    \qquad a>d_x/2,
\]
for a fixed radius \(c_a>0\) and \(d_{\rm out}\) components of the state,
with the \(L^2\) distance \(\dist_\Omega\), so that
\(\lVert\dist_\Omega(\widehat\Phi,\Phi)\rVert_{L^2(\rho)}^2
=\int\|\widehat\Phi(u)-\Phi(u)\|_{L^2}^2\dd\rho(u)\)
is the mean-square error of the learned map that is bounded in
\cite[Theorem~3.6]{chandramoorthy2026spectral}. Assume that \(\Phi\)
belongs to the class of operators \(\mathsf S^{\rm poly}\) of
\cite{chandramoorthy2026spectral}, that is, it satisfies
\cite[Condition~3.3]{chandramoorthy2026spectral} on \(\Ba\) through a stable
and accurate semi-implicit Fourier spectral scheme with polynomial
nonlinearity, and that \(\operatorname{supp}\rho\subset\Ba\), so that
\cite[Theorem~3.6]{chandramoorthy2026spectral} applies to the corresponding
FNO class. The estimator
\(\widehat\Phi\) is the corresponding constrained empirical risk minimizer
trained on exact pairs \((u_i,\Phi(u_i))\) with \(u_i\) drawn independently
from \(\rho\). The standing assumption of \cref{sec:kbe} that
\(Tf=f\circ\Phi\) is well defined and bounded on \(L^2(\rho)\) remains in
force. It is not implied by the hypotheses of the cited theorem, which allow
any input measure supported in \(\Ba\); it holds, for instance, if \(\rho\)
is invariant under \(\Phi\), and more generally if the image measure of
\(\rho\) under \(\Phi\) has a bounded density with respect to \(\rho\). The
candidate observables are assumed to be Lipschitz on a bounded set
\(\Omega_0\subset\Omega\) that contains \(\Phi(\operatorname{supp}\rho)\)
and the values on \(\operatorname{supp}\rho\) of every map in the FNO
class, so that \(\Lip(V_\theta)<\infty\) does not depend on
\(\widehat\Phi\).

\begin{remark}[Choice of the level]\label{rem:level}
A larger level of the candidate family is not always better, because the
learning and sampling errors in \cref{thm:main} must be controlled
uniformly over richer families, and a sufficiently rich family can make
\(\widehat\cJ\) small on the sample without \(V\) being nearly invariant.
By \eqref{eq:sampling-norm} and \cref{lem:basis-linear-surrogate},
\(\Delta(V_\theta)\le(m(1+\eta_{V_\theta}))^{1/2}
\max_j\|\widehat T\widehat\psi_j-T\widehat\psi_j\|\) for the computational
basis of \(V_\theta\), so that \(\Delta_\Theta\) is small whenever
\(\widehat T\) is uniformly accurate on the candidate basis functions. For
composition estimators, the level enters through the Lipschitz constant.
With the Gram matrix \(G_{\phi_\theta}\) of \(\phi_\theta\) in
\(L^2(\rho)\),
\(\Lip(V_\theta)\le\Lip(\phi_\theta)\,\lambda_{\min}(G_{\phi_\theta})^{-1/2}\),
and for polynomial dictionaries of degree \(p\) on a bounded interval that
contains \(\Omega_0\) and on which \(\rho\) has a density bounded below,
\(\Lip(V_\theta)\) is at most of order \(p^3\) by the Markov and Nikolskii
inequalities \cite{borwein1995polynomials}. In the setting of
\cref{cor:fno-rate}, \(\Delta_{\Theta_p}\) therefore tends to zero in
probability along levels \(p=p(\nop)\) for which
\(\Lip_{\Theta_p}\,\nop^{-1/(12+4d_x/a)}\to0\) as \(\nop\to\infty\).
\end{remark}

\paragraph{A posteriori estimate}
For a general approximation, the learning error can be estimated from held-out
data. Conditionally on \(\widehat T\) and on the sample points, let
\(X\sim\rho\) be a state with successor \(X^+\), independent of both, and
let \(f\in\Fn\) satisfy \(\EE\,|f(X^+)|^2<\infty\). Then, as in
\cref{rem:stochastic-residual-bias},
\[
    \EE\,|(\widehat Tf)(X)-f(X^+)|^2
    =\|(\widehat T-T)f\|^2+\EE\,|f(X^+)-(Tf)(X)|^2 .
\]
Summed over the computational basis \(\widehat\Psi\) of \(V\), the expected
squared prediction error of \(\widehat T\) on held-out data therefore
equals
\(\sum_j\|(\widehat T-T)\widehat\psi_j\|^2
+\sum_j\EE\,|\widehat\psi_j(X^+)-(T\widehat\psi_j)(X)|^2\). The first sum
is at least \(\Delta(V)^2/(1+\eta_V)\) by \cref{lem:basis-linear-surrogate}
and \eqref{eq:sampling-norm}, so that the prediction error bounds
\(\Delta(V)^2\) from above up to this factor; the second sum, the
conditional variance, vanishes for deterministic dynamics.

\section{Implementation details of the numerical experiments}
\label{app:numerical-details}

\paragraph{Ornstein--Uhlenbeck process}
The backward equation is discretized on 128 points in \([-5,5]\). The
training set contains 384 observables, each a linear combination of
Gaussian-weighted Hermite polynomials, Fourier modes, and Gaussian
functions, together with their images at \(\tau=0.15\). The
FNO uses 24 Fourier modes, width 96, four layers, and 3000 epochs. The
Galerkin space consists of the constant function, which is held fixed, and
19 functions from a \(\tanh\) neural dictionary, followed by weighted
orthonormalization. The displayed
eigenfunctions are normalized and sign-aligned, with polynomial projection
and parity postprocessing used only for visualization.

\paragraph{Triple-well diffusion}
The benchmark follows \cite{xu2025koopmansemigroup}; the potential is the
three-hole potential of \cite{metzner2006illustration},
\[
\begin{aligned}
U(x,y)\coloneqq{}&3e^{-x^2-(y-1/3)^2}-3e^{-x^2-(y-5/3)^2}
          -5e^{-(x-1)^2-y^2}-5e^{-(x+1)^2-y^2}\\
        &+0.2x^4+0.2(y-1/3)^4.
\end{aligned}
\]
The scalar fields are represented on a \(64\times64\) grid over
\([-2,2]\times[-1,2]\). Sixty-four random Fourier--Gaussian observables are
used for training. The FNO uses \(16\times16\) modes, width 96, five layers,
and 500 epochs. The 49-dimensional Galerkin space contains a constant and 48
neural functions. Reference rates are computed from the finite-difference generator on a \(128\times128\) grid. The displayed modes use one-to-one Hungarian matching, symmetry constraints, and smoothing.

\paragraph{Lorenz systems}
The first 12000 sampled times are used for fitting and the last 3000 for
evaluation. Lorenz--63 uses \(\tau=0.1\) and a 140-dimensional
dictionary containing the standardized sensor curve and degree-two through
degree-four monomials of five PCA coordinates. Lorenz--96 uses \(\tau=0.02\)
and a 30-dimensional dictionary containing the standardized field and
quadratic monomials of four PCA coordinates. The spectral radii of the fitted Galerkin matrices are
0.958 (Lorenz--63) and 0.981 (Lorenz--96), so the configured upper bounds on
the spectral radius (radial caps) are inactive.

\paragraph{Sea-surface temperature}
The 11 input fields are listed in \cref{tab:sst-fields}. The ERA5
fields are daily means of the 6-hourly data on the WeatherBench~2 product
conservatively regridded to the \(32\times64\) grid. The wind stress is computed
at each 6-hourly step by the bulk formula
\(\rho_a C_D\,|\mathbf u_{10}|\,\mathbf u_{10}\), where \(\mathbf u_{10}\) is
the 10\,m wind, \(\rho_a=1.225\) kg\,m\(^{-3}\), and \(C_D=1.3\times10^{-3}\),
and then averaged over the day. The GODAS fields are monthly means, interpolated bilinearly to
the same grid and held constant within each month; to keep the inputs causal,
the value for a given month is used from the first day of the following month
onward. Potential temperature is converted from K to \(^\circ\)C and salinity
from kg\,kg\(^{-1}\) to g\,kg\(^{-1}\); all fields are then standardized with
fitting-window statistics.
\begin{table}[tbp]
\centering
\caption{Input fields of the SST experiment on the \(32\times64\) grid. WB2:
WeatherBench~2 archive of ERA5. Units are those of the fields as used, before
standardization. Only SST is verified.}
\label{tab:sst-fields}
\footnotesize
\setlength{\tabcolsep}{3.5pt}
\begin{tabular}{llll}
\hline
Field & Source (variable) & Level / depth & Unit \\
\hline
SST (target) & WB2 (ERA5 SST) & surface & K \\
Zonal wind stress & WB2 (from 10\,m winds) & surface & N\,m\(^{-2}\) \\
Meridional wind stress & WB2 (from 10\,m winds) & surface & N\,m\(^{-2}\) \\
Sea-surface height & GODAS (\texttt{sshg}) & surface & m \\
Potential temperature & GODAS (\texttt{pottmp}) & 55\,m & \(^\circ\)C \\
Potential temperature & GODAS (\texttt{pottmp}) & 105\,m & \(^\circ\)C \\
Zonal current & GODAS (\texttt{ucur}) & 5\,m & m\,s\(^{-1}\) \\
Meridional current & GODAS (\texttt{vcur}) & 5\,m & m\,s\(^{-1}\) \\
Mixed-layer depth & GODAS (\texttt{dbss\_obml}) & -- & m \\
Salinity & GODAS (\texttt{salt}) & 5\,m & g\,kg\(^{-1}\) \\
Net surface heat flux (downward) & GODAS (\texttt{thflx}) & surface & W\,m\(^{-2}\) \\
\hline
\end{tabular}
\end{table}

SST is centered by a training-only leap-aligned calendar-day climatology,
and each auxiliary field is deseasonalized at every grid point by regressing
on a constant and six annual harmonics (13 regression columns); the
harmonics are not appended to the state. The reduction is channelwise.
Cosine-latitude-weighted empirical orthogonal functions (EOFs) of the
standardized anomalies are computed on
the fitting window only, SST anomalies are represented by 64 EOF coefficients
scaled by \(\sqrt{0.5}\), and each of the 10 auxiliary fields by 4
coefficients scaled by \(\sqrt{0.05}\). The resulting \(64+4\times10=104\)
coordinates are centered and projected by a truncated SVD onto a joint
rank-96 PCA space with coefficient vector \(\mathbf z\in\mathbb R^{96}\).
The reduced matrix is fitted by ridge regression, without intercept and with
equal weights over the 9860 training days, to the blended one-day target
\[
    \mathbf z_{\rm target}=0.6\,\mathbf z_{\rm observed}+0.4\,\mathbf z_{\rm FNO},
\]
where \(\mathbf z_{\rm observed}\) is the observed next-day coefficient vector and
\(\mathbf z_{\rm FNO}\) is obtained by applying the fixed FNO to the standardized
11-channel state and encoding its one-day prediction with the target-day
climatology, the EOFs, and the same joint projection. The ridge parameter is
\(10^{-2}\) times the mean diagonal entry of the Gram matrix
(about \(0.70\)). Two stability adjustments are then applied. First,
every diagonal block of the real Schur form of the fitted matrix whose
spectral radius exceeds 0.9975 is scaled to that radius, keeping the
off-diagonal couplings; here the raw spectral radius is 0.9978 and exactly
one \(1\times1\) block is affected. Second, the launch-time SST residual
\(\xi_0\), i.e., the part of the SST anomaly not reconstructed by the
joint decoding, is propagated as \(c_{\rm res}^{\ell}\xi_0\) at lead \(\ell\) days and
added back to the reconstruction, with a single scalar
\(c_{\rm res}=0.905\) estimated by area-weighted least squares on consecutive
training residuals.

The FNO and SFNO use 11 channels, four layers, width 96, spectral modes
\((12,24)\), and are trained for 180 epochs (no early stopping) with AdamW,
initial learning rate \(5\times10^{-4}\), weight decay \(10^{-5}\), a cosine
schedule, and batch size 12. The loss is the relative \(L^2\) error of the
one-day prediction plus \(0.25\) times that of the two-day prediction
obtained by applying the network twice. With probability \(0.5\) per batch,
four training triples are replaced by their random linear combination with a
common unit-norm Gaussian coefficient vector, applied identically at the three
times. In the 540-day test, finiteness is checked on the full 11-channel
standardized state for FNO and SFNO and on the decoded SST for BSO; SFNO
first becomes nonfinite at day 171, whereas FNO and BSO remain finite through
day 540.

\paragraph{EEG}
The first 60 seconds of 64-channel EEG sampled at 160 Hz are common-average
referenced, causally bandpass filtered to 8--30 Hz, and standardized using
statistics from 1 to 15 seconds, after a one-second filter warm-up. The state
is a \(64\times160\) one-second delay field advanced by \(\tau=0.1\) seconds
(16 samples). Fitting pairs use delay fields ending at times spaced 0.125
seconds (20 samples) apart, from 2 to 14.875 seconds, which gives 104 pairs
with strongly overlapping histories. BSO learns the delay-field map
with a temporal FNO and uses four training-only PCA coordinates, a constant,
and 20 learned functions, producing a \(25\times25\) ridge-regularized
matrix.

\bibliographystyle{siamplain}
\bibliography{references}
\end{document}